\documentclass[12pt]{amsart}
\usepackage{graphicx}
\graphicspath{{../Figures/}}
\usepackage{amssymb}
\usepackage{bm}

\allowdisplaybreaks

\newtheorem{theorem}{Theorem}
\newtheorem{lemma}[theorem]{Lemma}
\newtheorem{proposition}[theorem]{Proposition}

\theoremstyle{definition}
\newtheorem{definition}[theorem]{Definition}

\theoremstyle{remark}
\newtheorem{remark}[theorem]{Remark}

\newtheorem{notation}[theorem]{Notation}

\newcommand{\thebottomline}{%
\renewcommand{\thefootnote}{}
\renewcommand{\footnoterule}{}
\phantom{M}\footnotetext{\hfill\tiny\textit{\noindent\romannumeral\day.\romannumeral\month.\romannumeral\year}}}

\newcommand{\str}{\mathrm{T}}
\newcommand{\ltr}{{\reflectbox{\tiny$\Gamma$}}}
\newcommand{\rtr}{\Gamma}

\newcommand{\tr}{\operatorname{tr}}
\newcommand{\E}{\operatorname{E}}

\newcommand{\bC}{\mathbb{C}}

\newcommand{\bR}{\mathbb{R}}

\newcommand{\bZ}{\mathbb{Z}}

\newcommand{\myvee}{\,\vee\,}

\let\phi=\varphi

\newcommand{\Tr}{\operatorname{Tr}}

\newcommand{\cA}{\mathcal{A}}
\newcommand{\ds}{\displaystyle}
\newcommand{\cP}{\mathcal{P}}

\theoremstyle{definition}

\newcommand{\ab}{\allowbreak}

\def\bR{\mathbb{R}}

\title[{Freeness and the Partial Transpose}] {Freeness and
  The Partial Transposes of Wishart Random Matrices}

\author[mingo]{James A. Mingo$^{(*)}$} \address{Department
  of Mathematics and Statistics, Queen's University, Jeffery
  Hall, Kingston, Ontario, K7L 3N6, Canada}

\email{mingo@mast.queensu.ca} 

\thanks{$^*$ Research supported by a Discovery Grant from
  the Natural Sciences and Engineering Research Council of
  Canada}

\author[popa]{Mihai Popa$^{( * )(* *)}$ } \address{The
  University of Texas at San Antonio, Department of
  Mathematics, One UTSA Circle, San Antonio, Texas 78249,
  and \newline ${}\hspace{.4cm}{}$ Institute of Mathematics
  ``Simion Stoilow'' of the Romanian Academy, P.O. Box
  1-764, Bucharest, RO-70700, Romania }

\email{Mihai.Popa@utsa.edu} 

\thanks{$^{ * * }$ Research supported by the Simons
  Foundation grant No. 360242.}

\begin{document}

\begin{abstract}
We show that the partial transposes of complex Wishart
random matrices are asymptotically free. We also investigate
regimes where the number of blocks is fixed but the size of
the blocks increases. This gives a example where the partial
transpose produces freeness at the operator level. Finally
we investigate the case of real Wishart matrices. 
\end{abstract}

\maketitle


\section{Introduction}\label{sec:introduction} 

Suppose we have a matrix $A$ in $M_{d_1}(\bC) \otimes
M_{d_2}(\bC)$. We can write this as a block matrix.
\[
A =
\def\arraystretch{1.4}
\left(\begin{array}{c|c|c}
A(1,1)   & \ \cdots\  & A(1,d_1) \\ \hline
\vdots   &            & \vdots  \\ \hline
A(d_1,1) & \ \cdots\  & A(d_1,d_1). \\
\end{array}\right).
\]
with each $A(i,j) \in M_{d_2}(\bC)$. We can form two partial
transposes of this matrix.

\hbox{
$\kern-1.5em
A^\ltr =
\def\arraystretch{1.4}
\left(\begin{array}{c|c|c}
A(1,1)   & \ \cdots\  & A(d_1,1) \\ \hline
\vdots   &            & \vdots  \\ \hline
A(1,d_1) & \ \cdots\  & A(d_1,d_1). \\
\end{array}\right),
A^\rtr =
\def\arraystretch{1.4}
\left(\begin{array}{c|c|c}
A(1,1)^\str   & \ \cdots\  & A(1,d_1)^\str \\ \hline
\vdots        &            & \vdots  \\ \hline
A(d_1,1)^\str & \ \cdots\  & A(d_1,d_1)^\str. \\
\end{array}\right).
$}

In quantum information theory the partial transpose has been
used as an entanglement detector. Suppose that $A$ is a 
positive matrix in $M_{d_1}(\bC) \otimes M_{d_2}(\bC)$. Recall 
that $A$ is entangled if  we cannot find positive matrices $B_1, \dots, \ab
B_k \in M_{d_1}(\bC)$ and $C_1, \dots, C_k \in M_{d_2}(\bC)$
such that $A = \sum_{i=1}^k A_i \otimes B_i$. If a positive
matrix fails to have a positive partial transpose then the
matrix must be entangled. It was shown by Aubrun
\cite{aubrun} that in a particular regime of the Wishart
distribution, the partial transpose of a positive matrix is
typically entangled. He also showed that the limiting
distribution of a partially transposed Wishart matrix is
semi-circular. This was quite unexpected. We revisit his
theorem and show that the conclusion can be explained by the
requirement that the non-crossing partitions that survive in
the limit, have to remain non-crossing when the order of
elements is reversed.

In this paper we show that in addition to transforming a
Marchenko-Pastur into a semi-circular distribution, the
partial transpose also produces freeness --- in two
different regimes. The first is when both $d_1$ and $d_2$
tend to $\infty$; we show that when $W$ is a Wishart matrix,
$W$, $W^\ltr$, $W^\rtr$ and $W^\str$ are asymptotically free
in the complex case and that $W = W^\str$ and $W^\rtr =
W^\ltr$ are asymptotically free in the real case.

The second regime is the one considered by Banica and
Nechita \cite{bn} where we fix $d_1$ and let $d_2$ tend to
$\infty$. In this case we show that $W$ and $W^\rtr$ are
asymptotically free. As $(W^\ltr)^\str = W^\rtr$, one can
has that $W^\ltr$ and $W^\rtr$ have the same limit
distributions. Banica and Nechita showed that the limit
distribution of $d_1W^\ltr$ could be written as the free
difference of two Marchenko-Pastur distributions.  We can
show that in the same regime and when $d_1 = 2$, one can
write the limit distribution of $d_1W^\ltr$ as the sum of
two free operators $X_1$ and $X_2$ coming from the diagonal
and off-diagonal parts of $d_1W^\ltr$. More precisely if we
write a Marchenko-Pastur random variable $w$ as a block
matrix
\[
w= \frac{1}{2} 
\begin{pmatrix}
w_{11} & w_{12} \\
w_{21} & w_{22}
\end{pmatrix},
X_1 = 
\begin{pmatrix}
w_{11} &    0   \\
0      & w_{22} \\
\end{pmatrix}
,
X_1 = 
\begin{pmatrix}
0 &    w_{21}   \\
w_{12}      & 0 \\
\end{pmatrix},
\]
then $d_1w^\ltr = X_1 + X_2$ and $X_1$ and $X_2$ are free.
Moreover $X_1$ is a Marchenko-Pastur random variable with
the same distribution as $2w$ and $X_2$ is an even operator
whose even cumulants are the same as those of $X_1$. So by
writing $d_1w^\ltr$ as a sum as opposed to a difference we
get a natural representation for the two operators.

The connection of the transpose to freeness goes back to the
work of Emily Redelmeier on real second order freeness
\cite{r2}. She showed that the fluctuation moments of real
Gaussian and Wishart matrices require the transpose to taken
into account. Later in \cite{mp2} the authors showed that
the transpose can also appear at the first order level.
Namely, that for many complex ensembles a random matrix
could be asymptotically free from its transpose. Before this
the known examples of asymptotic freeness required
independence of the entries, see \cite[Ch. 4]{ms} and
\cite[Lect. 23, 24]{ns} for examples.

Our main tool here is the explicit evaluation mixed moments
of the four matrices $W$, $W^\ltr$, $W^\rtr$ and $W^\str$
and then to show that mixed cumulants must vanish, thus
demonstrating freeness. In order to achieve this we use the
technique of doubling of indices which already appeared in
the work of Redelmeier \cite{r1, r2}.

Besides the transpose one can consider the action of other
positive linear maps on the blocks of matrices and the effect
on the limit eigenvalue distribution. This was considered in 
considerable generality in the recent paper of Arizmendi, 
Nechita, and Vargas \cite{anv}. A third regime was considered 
by Fukuda and \'Sniady in \cite{fs} and a connection to meander 
polynomials was found.

The outline of the paper is as follows. In section
\ref{sec:notation} we establish the notation needed for our
calculations. The main method for computing mixed moments is
to expand the expression as a sum over the symmetric
group. This part is presented in Theorem
\ref{thm:exp_mixed_mom} of Section
\ref{sec:basic-calculations}. In Section
\ref{sec:asymptotics-of-permutations} we determine which
permutations contribute to the limit. The main result in
this section is Proposition
\ref{prop:sigma_non_crossing}. In Section
\ref{sec:limit-distributions} we consider the limit
distributions of our partially transposed operators in the
two regimes. In particular we shall show that the operators
$X_1$ and $X_2$ mentioned above are free. In Section
\ref{sec:asymptotic-freeness} we present our main results,
the asymptotic freeness of our partially transposed Wishart
matrices. In Section \ref{sec:real-case} we consider the
situation for real Wishart matrices.


\section{Notation and statement of results}\label{sec:notation}

Suppose $G_1, \dots, G_{d_1}$ are $d_2 \times p$ random
matrices where $G_i = (g^{(i)}_{jk})_{jk}$ and
$g^{(i)}_{jk}$ are complex Gaussian random variables with
mean 0 and (complex) variance 1, i.e. $\E(|g^{(i)}_{jk}|^2)
= 1$. Moreover suppose that the random variables $\{
g^{(i)}_{jk} \}_{i,j,k}$ are independent.

\[
W =  \frac{1}{d_1d_2}
\left(
\begin{array}{c}
G_1 \\ \hline
\vdots \\ \hline
G_{d_1}
\end{array}
\right)
\left(
\begin{array}{c|c|c}
G_1^* & \cdots & G_{d_1}^*
\end{array}
\right) = \frac{1}{d_1d_2} (G_iG_j^*)_{ij}
\]
is a $d_1d_2 \times d_1d_2$ \textit{complex} Wishart matrix. We write $W =
d_1^{-1}(W(i,j))_{ij}$ as $d_1 \times d_1$ block matrix with
each entry the $d_2 \times d_2$ matrix
$d_2^{-1}G_iG_j^*$. From this we get four matrices $\{ W,
W^\ltr, W^\rtr, W^\str\}$ defined as follows:

\begin{itemize}

\item
$W^\str = \frac{1}{d_1}(W(j,i)^\str)_{ij}$ is
  the ``full'' transpose

\item
$W^\ltr = \frac{1}{d_1}(W(j,i))_{ij}$ is the ``left''
  partial transpose

\item
$W^\rtr = \frac{1}{d_1}(W(i,j)^\str)_{ij}$ is
  the ``right'' partial transpose
\end{itemize}

Note that the $X \mapsto X^\rtr$ notation conceals the
dependence on $d_1$ and $d_2$. Thus as the size of the
matrix grows these operators might be expected to behave
differently, depending on the way $d_1$ and $d_2$ grow.

If the random variables $\{ g{(i)}_{jk}\}_{i,j, k}$ are real
Gaussian random variables with mean 0 and variance 1 then
$W$ is a \textit{real} Wishart matrix. For many eigenvalue
results there is no distinction between the real and complex
case. In \cite{mp2} we showed that when it comes to freeness
there is a difference, in particular with respect to the
behaviour of the transpose. In this paper we show that with
the partial transpose we continue to see a difference
between the real and complex cases.

If we assume that $d_1, d_2 \rightarrow \infty$ and that
$\ds\frac{p}{d_1d_2}\rightarrow c$, $0 < c < \infty$, then
the eigenvalue distributions of $W$ and $W^\str$ converge to
Marchenko-Pastur with parameter $c$. This is the
distribution on $\bR^+$ that has density $\ds\frac{\sqrt{(b
    - t)(t - a)}}{2 \pi t}$ on $[a, b]$ and an atom of
weight $(1 - c)$ at $0$ if $ c < 1$; we set $b = (1 +
\sqrt{c})^2$ and $a = (1 - \sqrt{c})^2$.

Note that we are using what one might call the free
probabilist's Marchenko-Pastur law. In our normalization all
the cumulants are equal to $c$. For the relation between the
two see \cite[Ch. 2 Remark 12]{ms}. With this normalization
we can restate Aubrun's theorem.

\medskip\noindent\textit{%
Suppose $\ds\lim_{d_1, d_2 \rightarrow \infty}
\frac{p}{d_1d_2} = c$, then the eigenvalue distributions of
$W^\ltr$ and $W^\rtr$ converge to a shifted semi-circular
operator with mean $c$ and variance $c$.
}

\medskip

Our main results are the following.

\begin{theorem}
Suppose $\ds\lim_{d_1, d_2 \rightarrow \infty}
\frac{p}{d_1d_2} = c$, then the family $\{ W, W^\ltr,
W^\rtr,\ab W^\str\}$ is asymptotically free in the complex case
and the family $\{ W, W^\rtr\}$ is asymptotically free in
the real case.
\end{theorem}

\begin{theorem}
Suppose $d_1$ is fixed and$\ds\lim_{d_2 \rightarrow \infty}
\frac{p}{d_1d_2} = c$, then the family $\{ W, W^\rtr\}$ is
asymptotically free in the complex case.
\end{theorem}

In $M_{2}(\bC)$ let $\{ E_{ij}\}_{ij=1}^2$ be the standard
matrix units. For convenience of notation we shall write
$E_{ij}$ for $E_{ij} \otimes I_{d_2} \in M_2(\bC) \otimes
M_{d_2}(\bC)$.

\begin{theorem}
Suppose $d_1 = 2$, $p, d_2 \rightarrow \infty$ and
$\ds\lim_{d_2 \rightarrow \infty} \frac{p}{d_1d_2} =
c$. Then $\{ W, W^\ltr, E_{ij} \}_{i, = 1}^2$ has a limit
joint distribution, $\{w, w^\ltr, e_{ij} \}_{i,j=1}^2$, in a
non-commutative $*$-probability space $(\cA,
\phi)$. Relative to the matrix units $\{ e_{ij} \}_{i,j =
  1}^2$ we write
\[
w = \frac{1}{2} \begin{pmatrix}
w_{11} & w_{12} \\ w_{21} & w_{22} \end{pmatrix}
\mbox{\ and\ }
w^\ltr = \frac{1}{2} \begin{pmatrix}
w_{11} & w_{21} \\ w_{12} & w_{22} \end{pmatrix}.
\]
Then $X_1 = \begin{pmatrix} w_{11} & 0 \\ 0 &
  w_{22} \end{pmatrix}$ and $X_2 = \begin{pmatrix} 0 &
  w_{21} \\ w_{12} & 0 \end{pmatrix}$ are free. $X_1$ has a
Marchenko-Pastur distribution with parameter $2c$ and $X_2$
is an even operator with all even cumulants equal to $2c$.
\end{theorem}


\section{A general formula for mixed moments}
\label{sec:basic-calculations}

Let $A_1, \dots, A_n$ be $N \times N$ matrices, then
\begin{equation}\label{eq:trace_expansion}
\Tr(A_1 \cdots A_n) =
\sum_{i_{\pm 1}, \dots, i_{\pm n}= 1}^N
a^{(1)}_{i_1, i_{-1}} a^{(2)}_{i_2, i_{-2}} \cdots
a^{(n)}_{i_n i_{-n}}
\end{equation}
where the sum runs over all $i: [\pm n] \rightarrow [N]$
such that $i(-1) = i(2)$, $i(-2) = i(3)$, \dots, $i(-n) = i(1)$. 

We wish to use the symmetric group $S_n$ of permutations on
$[n] = \{1, 2, 3, \dots, n\}$ to keep track of the partial
transposes. So we shall introduce the following notation.
Given a permutation $\sigma$ in $S_n$ we extend $\sigma$ to
be a permutation on $[\pm n] = \{1, -1, 2, -2, 3, -3, \dots,
n, -n\}$ be setting $\sigma(-k) = -k$ for $k > 0$. We let
$\delta$ be the permutation of $[\pm n]$ given by $\delta(k)
= -k$ for all $k \in [\pm n]$ and $\gamma \in S_n$ be the
permutation with one cycle: $\gamma = (1, 2, 3, \dots,
n)$. With our conventions we have $\gamma \delta \gamma^{-1}
= (1, -n)(2, -1)(3, -2) \cdots (n, -(n-1))$. The condition
in (\ref{eq:trace_expansion}) now becomes $i = i\circ
\gamma\delta \gamma^{-1}$.

To show that the family $\{ W, W^\ltr, W^\rtr, W^\str\}$ is
asymptotically free we shall have to compute the expectation
of the trace of arbitrary words in $\{ W, W^\ltr, W^\rtr,
W^\str\}$. For this we use the following notation.

Let $(\epsilon, \eta) \in \{-1, 1\}^2 = \bZ_2^2$. 

\medskip
\noindent
We set $W^{(\epsilon, \eta)} = \left\{
\def\arraystretch{1.3}\begin{array}{cl}
W & \mbox{\ if\ } (\epsilon, \eta) = (1,1) \\
W^\ltr & \mbox{\ if\ } (\epsilon, \eta) = (-1,1) \\
W^\rtr & \mbox{\ if\ } (\epsilon, \eta) = (1,-1) \\
W^\str & \mbox{\ if\ } (\epsilon, \eta) = (-1,-1) \\
\end{array}\right.$

\medskip

Let $(\epsilon_1, \eta_1), \dots, (\epsilon_n, \eta_n) \in
\bZ_2^n$, then an arbitrary word in $\{ W, W^\ltr,
W^\rtr,\ab W^\str\}$ is $W^{(\epsilon_1, \eta_1)} \cdots
W^{(\epsilon_n, \eta_n)}$ and we seek to write
$\ds\lim_{d_1, d_2 \rightarrow \infty}
\E(\tr(W^{(\epsilon_1, \eta_1)}\ab \cdots \ab
W^{(\epsilon_n, \eta_n)}))$ as a sum of free cumulants.

To achieve this we need to introduce still more notation. We
shall suppose that $n$, the length of the word, is fixed for
the moment. Now given $(\epsilon_1, \epsilon_2, \dots,
\epsilon_n)$ we denote by $\epsilon$ the permutation of
$[\pm n]$ given by $\epsilon(k) = \epsilon_{|k|} k$; here $k
\in [\pm n]$, but $|k| > 0$, so $\epsilon_{|k|}$ means the
$k^{th}$ element of our vector $(\epsilon_1, \dots,
\epsilon_n)$. Similarly given $(\eta_1, \dots, \eta_n)$ we
get the permutation $\eta$ of $[\pm n]$. Note that $\delta$,
$\epsilon$ and $\eta$ all commute with each other.

We shall think of $W$, $W^\ltr$, $W^\rtr$, and $W^\str$ as
random elements of $M_{d_1}(\bC) \otimes M_{d_2}(\bC)$. On
this algebra we have a trace $\Tr \otimes \Tr$; we also have
the normalized trace $\tr \otimes \tr = \ds\frac{1}{d_1d_2}
\Tr \otimes \Tr$.

\begin{remark}
With the notations above, we have that
\begin{eqnarray}\label{eq:expansion}\lefteqn{%
\E(\Tr \otimes \Tr(W^{(\epsilon_1, \eta_1)} \cdots
W^{(\epsilon_n, \eta_n)}))} \notag \\ & = & (d_1d_2)^{-n}
  \kern-1em \sum_{j_{\pm 1}, \dots, j_{\pm n}} \sum_{s_{\pm
      1}, \dots, s_{\pm n}} \sum_{t_1, \dots, t_n} \E(
  g^{(j_1)}_{s_1t_1} \cdots g^{(j_n)}_{s_nt_n}
  \ \overline{g^{(j_{-1})}_{s_{-1}t_1}} \cdots
  \overline{g^{(j_{-n})}_{s_{-n}t_n}})
\end{eqnarray}
where the summation is subject to the conditions that $j = j
\circ\epsilon \gamma\delta \gamma^{-1} \epsilon$, $s = s
\circ \eta\gamma\delta \gamma^{-1}\eta$.
\end{remark}

\begin{proof}

\begin{eqnarray*}\lefteqn{
\Tr \otimes \Tr\big( W^{(\epsilon_1, \eta_1)} \cdots
W^{(\epsilon_n, \eta_n)}\big)}\\ &=& d_1^{-n}\sum_{i_1,
    \dots, i_n=1}^{d_1} \Tr\Big( \big(W^{(\epsilon_1,
    \eta_1)}\big)_{i_1i_2} \cdots \big(W^{(\epsilon_n,
    \eta_n)})_{i_ni_1}\Big)\\ &=& d_1^{-n}
  \mathop{\sum_{i_{\pm 1}, \dots, i_{\pm n}=1}}_{i = i\circ
    \gamma\delta \gamma^{-1}}^{d_1} \Tr\Big(
  \big(W^{(\epsilon_1, \eta_1)}\big)_{i_1i_{-1}} \cdots
  \big(W^{(\epsilon_n, \eta_n)})_{i_ni_{-n}}\Big) \\ &=&
  d_1^{-n} \mathop{\sum_{j_{\pm 1}, \dots, j_{\pm n}=1}}_ {j
    = j \circ \epsilon \gamma\delta\gamma^{-1} \epsilon}
  ^{d_1} \Tr\Big( W(j_1, j_{-1})^{(\eta_1)} \cdots
  W(j_n,j_{-n}) ^{(\eta_n)}\Big).
\end{eqnarray*} 
To achieve the last step we let $j = i \circ \epsilon$. Also
we have adopted the convention that for $A$ a matrix in
$M_{d_2}(\bC)$ we let $A^{(1)} = A$ and $A^{(-1)} = A^\str$.

Next we must expand $\Tr\Big(
W(j_1, j_{-1})^{(\eta_1)} \cdots 
W(j_n,j_{-n}) ^{(\eta_n)}\Big)$. 

\begin{eqnarray*}\lefteqn{
\Tr\big( W(j_1, j_{-1})^{(\eta_1)} \cdots 
W(j_n, j_{-n})^{(\eta_n)}\big) } \\
& = &
\mathop{\sum_{r_{\pm 1}, \dots , r_{\pm n}=1}}_
{r = r \circ \gamma\delta\gamma^{-1}}^{d_2}
\big(W(j_1, j_{-1})^{(\eta_1)})_{r_1r_{-1}} \cdots 
\big(W(j_n, j_{-n})^{(\eta_n)}\big)_{r_nr_{-n}} \\
&=&
\mathop{\sum_{s_{\pm 1}, \dots , s_{\pm n}=1}}_
{s = s \circ \eta\gamma\delta\gamma^{-1}\eta}^{d_2}
\big(W(j_1, j_{-1})\big)_{s_1s_{-1}} \cdots 
\big(W(j_n, j_{-n})\big)_{s_ns_{-n}} \\
& = & d_2^{-n}
\sum_{s_{\pm 1}, \dots , s_{\pm n}}
\big(G_{j_1}G_{j_{-1}}^*\big)_{s_1s_{-1}} \cdots
\big(G_{j_n}G_{j_{-n}}^*\big)_{s_ns_{-n}} \\
& = & d_2^{-n}
\mathop{\sum_{s_{\pm 1}, \dots , s_{\pm n}}}
\sum_{t_1, \dots, t_n=1}^p
g^{(j_1)}_{s_1t_1} \overline{g^{(j_{-1})}_{s_{-1}t_1}}
\cdots
g^{(j_n)}_{s_nt_n} \overline{g^{(j_{-n})}_{s_{-n}t_n}}
\end{eqnarray*}
hence the conclusion.
\end{proof}

Next we need to compute $\E( g^{(j_1)}_{s_1t_1} \cdots
g^{(j_n)}_{s_nt_n} \ \overline{g^{(j_{-1})}_{s_{-1}t_1}}
\cdots \overline{g^{(j_{-n})}_{s_{-n}t_n}})$. We shall use
the complex form of Wick's rule that says that if $g_1,
\dots, g_m$ are independent Gaussian random variables and
$\alpha_1, \dots, \alpha_n, \beta_1, \dots, \beta_n \in [m]$
then $\E(g_{\alpha(1)} \cdots g_{\alpha(n)}
\overline{g_{\beta(1)}} \cdots \overline{g_{\beta(n)}})$ is
the number of permutations $\sigma \in S_n$ such that for
all $k \in [n]$ we have $\beta(k) = \alpha(\sigma(k))$; see
Janson \cite[page 13]{janson}.

Thus we let $g_{\alpha(k)} = g_{s_kt_k}^{(j_k)}$ and
$g_{\beta(k)} = g_{s_{-k}t_k}^{(j_{-k})}$. So if $\sigma \in
S_n$ and $\beta(k) = \alpha(\sigma(k))$ we have
\begin{equation}\label{eq:conditions}
\begin{cases}
s(-k) = s(\sigma(k)) & \textrm{for } k > 0   \\
j(-k) = j(\sigma(k)) & \textrm{for } k > 0   \\
t = t(\sigma(k))     & \textrm{for } k > 0
\end{cases}
\end{equation}

We wish to write the first two conditions as an equation
involving functions on $[\pm n]$.

\begin{lemma}\label{lemma:composition}
Let $\sigma \in S_n$ and $j : [\pm n] \rightarrow [d]$.  We
have $j(-k) = j(\sigma(k))$ for all $k > 0$ if and only if
$j = j \circ \sigma\delta\sigma^{-1}$.
\end{lemma}

\begin{proof}
Suppose that for $k > 0$ we have $j(-k) =
j(\sigma(k))$. Then for $k > 0$ we have $j\circ \sigma
\delta \sigma^{-1}(k) = j(-\sigma^{-1}(k)) = j(-l) =
j(\sigma(l)) = j(k)$ where $l = \sigma^{-1}(k) > 0$. Also $j
\circ \sigma \delta \sigma ^{-1}(-k) = j(\sigma(k)) =
j(-k)$. Thus $j = j \circ \sigma \gamma \sigma^{-1}$.

Now suppose that $j = j \circ \sigma \delta
\sigma^{-1}$. For $k > 0$ we have that $j(-k) = j\circ
\sigma\delta \sigma^{-1}(-k) = j(\sigma(k))$ as claimed.
\end{proof}

\begin{lemma}\label{lemma:expectation}
$\E(g^{(j_1)}_{s_1t_1} \cdots g^{(j_n)}_{s_nt_n}
  \ \overline{g^{(j_{-1})}_{s_{-1}t_1}} \cdots
  \overline{g^{(j_{-n})}_{s_{-n}t_n}}) = |\{\sigma \in S_n
  \mid j = j \circ \sigma \delta \sigma^{-1}$, $s = s\circ
  \sigma \delta \sigma^{-1}$, and $t = t \circ \sigma\}|$.
\end{lemma}

\begin{proof}
By Lemma \ref{lemma:composition} and Equation
(\ref{eq:conditions}) we have to count the number of
permutations $\sigma$ such that $j = j \circ \sigma
\delta \sigma^{-1}$ on the set $[d_1]$, $s = s\circ \sigma
\delta \sigma^{-1}$ on the set $[d_2]$, and $t = t \circ \sigma$
on the set $[p]$.
\end{proof}

\begin{remark}
In the proposition below we use the notation $\#(\sigma)$ to
denote the number of cycles in the cycle decomposition of
$\sigma$. If $\sigma$ and $\pi$ are permutations we let
$\sigma \vee \pi$ be the partition obtained by regarding
$\sigma$ and $\pi$ as partitions where the blocks of the
partition are the cycles of the permutation. Now $\sigma
\vee \pi$ denotes the supremum of the two partitions in the
lattice of partitions. Recall that the function $\sigma
\mapsto \#(\sigma)$ is a central function on $S_n$, so
$\#(\pi^{-1}\sigma \pi) = \#(\sigma)$ for all $\pi$ and
$\sigma$.
\end{remark}

\begin{proposition}\label{prop:partial_exp}
Subject to the conditions $j = j \circ\epsilon \gamma\delta
\gamma^{-1} \epsilon$ and $s = s \circ \eta\gamma\delta
\gamma^{-1}\eta$,
\begin{eqnarray*}\lefteqn{%
\sum_{j_{\pm 1}, \dots, j_{\pm n}} \sum_{s_{\pm 1}, \dots,
  s_{\pm n}} \sum_{t_1, \dots, t_n} \E( g^{(j_1)}_{s_1t_1}
\cdots g^{(j_n)}_{s_nt_n}
\ \overline{g^{(j_{-1})}_{s_{-1}t_1}} \cdots
\overline{g^{(j_{-n})}_{s_{-n}t_n}})}\\ & = & \sum_{\sigma
    \in S_n} d_1^{\#(\epsilon \gamma\delta \gamma^{-1}
    \epsilon \myvee   \sigma \delta \sigma^{-1})} d_2^{\#(\eta
    \gamma\delta \gamma^{-1} \eta \myvee \sigma \delta
    \sigma^{-1})} p^{\#(\sigma)}
\end{eqnarray*}
\end{proposition}

\begin{proof}
According to Lemma \ref{lemma:expectation}, subject to the
conditions $j = j \circ\epsilon \gamma\delta \gamma^{-1}
\epsilon$ and $s = s \circ \eta\gamma\delta \gamma^{-1}\eta$,
\begin{eqnarray*}\lefteqn{%
\sum_{j_{\pm 1}, \dots, j_{\pm n}} \sum_{s_{\pm 1}, \dots,
  s_{\pm n}} \sum_{t_1, \dots, t_n} \E( g^{(j_1)}_{s_1t_1}
\cdots g^{(j_n)}_{s_nt_n}
\ \overline{g^{(j_{-1})}_{s_{-1}t_1}} \cdots
\overline{g^{(j_{-n})}_{s_{-n}t_n}})} \\ 
& = & \kern-1em
\mathop{\mathop{\sum_{j_{\pm 1}, \dots, j_{\pm n}}}_%
    {s_{\pm 1}, \dots, s_{\pm n}}}_%
    {t_1, \dots, t_n}
|\{\sigma \in S_n \mid 
j = j \circ \sigma \delta\sigma^{-1}, 
s = s\circ \sigma \delta \sigma^{-1},
  \mathrm{\ and\ } 
t = t \circ \sigma\}| \\ 
& = &
\sum_{\sigma \in S_n} |\{ (j,s, t) \mid 
j = j \circ \sigma \delta \sigma^{-1}, 
s = s\circ \sigma \delta \sigma^{-1}, 
\mathrm{\ and\ } t = t \circ \sigma\}| \\ 
& = &
  \sum_{\sigma \in S_n} d_1^{\#(\epsilon \gamma\delta
    \gamma^{-1} \epsilon \myvee \sigma\delta\sigma^{-1})}
  d_2^{\#(\eta \gamma\delta \gamma^{-1} \eta \myvee
    \sigma\delta\sigma^{-1})} p^{\#(\sigma)}.
\end{eqnarray*}
To get the last equality we recall that the condition on $j$
is that it must simultaneously satisfy $j = j \circ\epsilon
\gamma\delta \gamma^{-1} \epsilon$ and $j = j \circ
\sigma \delta \sigma^{-1}$. So $j$ must be constant on the
cycles of $\epsilon \gamma\delta \gamma^{-1} \epsilon$ and
of $\sigma \delta \sigma^{-1}$; so $j$ must be constant on
the blocks of $\epsilon \gamma\delta \gamma^{-1} \epsilon
\vee \sigma \delta \sigma^{-1}$. The same argument applies to
$s$. The only condition on $t$ is that $ t = t \circ
\sigma$.
\end{proof}

\begin{theorem}\label{thm:exp_mixed_mom}
\[
\E(\tr \otimes \tr(W^{(\epsilon_1, \eta_1)} \cdots
W^{(\epsilon_n, \eta_n)}))
= 
\sum_{\sigma\in S_n}
\bigg(\frac{p}{d_1d_2}\bigg)^{\#(\sigma)}
d_1^{\ f_\epsilon(\sigma)}
d_2^{\ f_\eta(\sigma)}
\]
where $f_\epsilon(\sigma) = \#( \epsilon \gamma\delta
\gamma^{-1} \epsilon \vee \sigma \delta \sigma^{-1} ) +
\#(\sigma) - (n + 1)$ and $f_\eta(\sigma) = \#( \eta
\gamma\delta \gamma^{-1} \eta \vee \sigma \delta
\sigma^{-1})+ \#(\sigma) - (n + 1)$
\end{theorem}

\begin{proof}
According to Equation (\ref{eq:expansion}) and Proposition
\ref{prop:partial_exp} we have
\begin{eqnarray*}\lefteqn{%
\E(\tr \otimes \tr(W^{(\epsilon_1, \eta_1)} \cdots
W^{(\epsilon_n, \eta_n)}))} \\ & = &
  \Big(\frac{1}{d_1d_2}\Big)^{n + 1} \sum_{\sigma \in S_n}
  d_1^{\#(\epsilon \gamma\delta \gamma^{-1} \epsilon\myvee
    \sigma \delta \sigma^{-1})} d_2^{\#(\eta \gamma\delta
    \gamma^{-1} \eta \myvee \sigma\delta\sigma^{-1})}
  p^{\#(\sigma)} \\ & = & \sum_{\pi \in S_n} \Big(
  \frac{p}{d_1d_2}\Big)^{\#(\sigma)} d_1^{\#(\epsilon
    \gamma\delta \gamma^{-1} \epsilon \,\myvee\,
    \sigma\delta\sigma^{-1}) + \#(\sigma) - (n+1)} \\ & &
  \qquad\qquad\qquad\mbox{}\times d_2^{\#(\eta \gamma\delta
    \gamma^{-1} \eta \myvee \sigma\delta\sigma^{-1}) +
    \#(\sigma) - (n+1)} \\ & = & \sum_{\sigma\in S_n}
  \bigg(\frac{p}{d_1d_2}\bigg)^{\#(\sigma)}
  d_1^{\ f_\epsilon(\sigma)} d_2^{\ f_\eta(\sigma)}
\end{eqnarray*}
\end{proof}


\section{Asymptotics of Permutations}
\label{sec:asymptotics-of-permutations}

Theorem \ref{thm:exp_mixed_mom} gave us an expansion of
mixed moments of $\{W, W^\ltr, W^\rtr,\ab W^\str\}$ as a sum
over the symmetric group. We now have to determine which
permutations contribute to the limit. We shall show that for
all $\epsilon$ and all $\sigma$, $f_\epsilon(\sigma) \leq 0$
and determine for which $\sigma$ equality is achieved. Our
first goal is to show that $f_\epsilon(\sigma) < 0$ unless
$\epsilon$ is constant on the cycles of $\sigma$. Since
$\epsilon$ is arbitrary, whatever we show for $\epsilon$
will apply to $\eta$.

There is a fundamental equation that we shall frequently use
in what follows. Given a subgroup, $G$, of the group $S_n$
of permutations of $[n]$, we shall say that the subgroup
acts transitively on $[n]$ if given $k, l\in [n]$ we can
find $\rho \in G$ such that $\rho(k) = l$.

Given two permutations $\pi$ and $\sigma$ of $S_n$ such that
the subgroup generated $\pi$ and $\sigma$ acts transitively
there is a non-negative integer $g$ such that
\begin{equation}\label{eq:eulers_formula}
\#(\pi) + \#(\pi^{-1}\sigma)  + \#(\sigma) = n + 2(1 - g)
\end{equation}

Recall that a pairing of $[n]$ is a partition $\pi$ of $[n]$
with all blocks of size 2; note this implies that $n$ is
even. The set of all pairings of $[n]$ is denoted
$\cP_2(n)$. We shall also regard such a $\pi$ as the
permutation whose cycles are the blocks of $\pi$. In this
case $\pi$ has no fixed points and $\pi^2 = id$. In
\cite[Lemma 2]{mp} we proved the following.

\begin{lemma}\label{lemma:pairingproduct}
Let $\pi, \sigma \in \cP_2(n)$ be pairings and $(i_1, i_2,
\dots, i_k)$ a cycle of $\pi\sigma$. Let $j_r =
\sigma(i_r)$. Then $(j_k, j_{k-1}, \dots , j_1)$ is also a
cycle of $\pi\sigma$, and these two cycles are distinct;
$\{i_, \dots, i_k, j_1,\dots, j_k\}$ is a block of $\pi \vee
\sigma$ and all are of this form; $2\#(\pi \vee \sigma) =
\#(\pi\sigma)$. The cycle decomposition of $\pi\sigma$ can
be written $c_1 c'_1 \cdots c_k c'_k$ where $c'_i = \sigma
c_i^{-1} \sigma$. With this notation the blocks of $\pi \vee
\sigma$ are $c_i \cup c'_i$.
\end{lemma}

\begin{lemma}\label{lemma:epsilon_constant}
Let $\sigma \in S_n$ and $\epsilon \in \bZ_2^n$ be given,
then $\epsilon$ is constant on the cycles of $\sigma$ if and
only if $\epsilon \delta \sigma \delta \sigma^{-1} \epsilon
[n] = [n]$.
\end{lemma}

\begin{proof}
We begin by noting that for $k > 0$
\begin{equation}\label{eq:sigma_epsilon}
\epsilon \delta \sigma \delta \sigma^{-1} \epsilon (k) =
\begin{cases}
\epsilon \delta \sigma \delta \sigma^{-1}(k) & \epsilon_k = 1 \\
\epsilon \delta \sigma(k) & \epsilon_k = -1
\end{cases}
= \begin{cases}
\epsilon_{\sigma^{-1}(k)} \sigma^{-1}(k) & \epsilon_k = 1 \\
-\epsilon_{\sigma(k)} \sigma(k) & \epsilon_k = -1
\end{cases}.
\end{equation}
Suppose $\epsilon$ is constant on the cycles of
$\sigma$. Then 
\[
\epsilon \delta \sigma \delta \sigma^{-1} \epsilon (k) =
\begin{cases}
\sigma^{-1}(k) & \epsilon_k = 1 \\
\sigma(k) & \epsilon_k = -1 \\
\end{cases} \in [n].
\]

Conversely suppose that $\epsilon \delta \sigma \delta
\sigma^{-1} \epsilon [n] = [n]$. Then for $k > 0$ by
Equation (\ref{eq:sigma_epsilon})
\[
\epsilon \delta \sigma \delta \sigma^{-1} \epsilon (k)
= \begin{cases}
\epsilon_{\sigma^{-1}(k)} \sigma^{-1}(k) & \epsilon_k = 1 \\
-\epsilon_{\sigma(k)} \sigma(k) & \epsilon_k = -1
\end{cases}
\]
and thus $\epsilon_{\sigma(k)}$ and $\epsilon_k$ have the same sign. 
\end{proof}

\begin{lemma}\label{lemma:negativity_of_f}
Let $\sigma \in S_n$ and $\epsilon \in \bZ_2^n$ be given
then $f_\epsilon(\sigma) < 0$ unless $\epsilon$ is constant
on the cycles of $\sigma$.
\end{lemma}

\begin{proof}
Suppose $\epsilon$ is not constant on the cycles of
$\sigma$, then by Lemma \ref{lemma:epsilon_constant} we have
that $\epsilon \delta \sigma \delta \sigma^{-1} \epsilon
[n]$ meets $[-n]$. Both $\epsilon \gamma\delta \gamma^{-1}
\epsilon$ and $\sigma\delta\sigma^{-1}$ are pairings and as
permutations $\delta$ and $\epsilon$ commute. Thus by Lemma
\ref{lemma:pairingproduct}
\begin{eqnarray*}\lefteqn{%
2\#( \epsilon \gamma \delta \gamma^{-1} \epsilon \vee
\sigma \delta \sigma^{-1})} \\
& = &
\#(\epsilon \gamma \delta
\gamma^{-1} \epsilon \sigma \delta \sigma^{-1} ) =
\#(\gamma \delta \gamma^{-1} \delta \ \epsilon \delta
\sigma \delta \sigma^{-1} \epsilon) \\
& = &
\#( (\epsilon\delta \sigma^{-1}\delta \sigma\epsilon)^{-1} \gamma
\delta \gamma^{-1} \delta).
\end{eqnarray*}
Now $\#(\gamma \delta \gamma^{-1}\delta) = 2$ and
$\#(\epsilon \delta \sigma \delta \sigma^{-1} \epsilon) =
\#(\delta \sigma \delta \sigma^{-1})=
\#(\delta\sigma\delta) + \#(\sigma^{-1}) = 2
\#(\sigma)$. Hence by Equation (\ref{eq:eulers_formula})
there is $g \geq 0$ such that
\[
\#(\epsilon \delta \sigma^{-1} \delta \sigma \epsilon) +
\#( ( \epsilon \delta \sigma^{-1} \delta \sigma\epsilon)^{-1} 
\gamma \delta \gamma^{-1} \delta ) + 
\#(\gamma \delta \gamma^{-1}\delta) 
= 2n + 2(1 - g),
\]
and thus
\[ 
f_\epsilon(\sigma) = \#(\epsilon \gamma\delta\gamma^{-1}\epsilon
\vee \sigma\delta\sigma^{-1}) + \#(\sigma) - (n + 1) = -(g +
1) \leq -1.
\]
\end{proof}

\begin{lemma}\label{lemma:sigma_epsilon}
Suppose that $\sigma \in S_n$ and $\epsilon \in \bZ_2^n$ and
$ \epsilon$ is constant on the cycles of $\sigma$. Then
there is a permutation $\sigma_\epsilon \in S_n$ such that
$\epsilon \delta \sigma \delta \sigma^{-1} \epsilon = \delta
\sigma_\epsilon \delta \sigma_\epsilon^{-1}$. Moreover if
$\sigma = c_1 \cdots c_k$ is the cycle decomposition of
$\sigma$ then $\sigma_\epsilon = c_1^{\lambda_1} \cdots
c_k^{\lambda_k}$ where $\lambda_i$ is the (constant) value
of $\epsilon$ on the cycle $c_i$.
\end{lemma} 

\begin{proof}
In the proof of Lemma \ref{lemma:epsilon_constant} we showed
that when $\epsilon$ is constant on the cycles of $\sigma$
we have that for $k > 0$
\[
\epsilon \delta \sigma \delta \sigma^{-1} \epsilon (k)
=
\begin{cases}
\sigma^{-1}(k) & \epsilon_k = 1 \\
\sigma(k)      & \epsilon_k = -1
\end{cases}.
\]
Thus on a cycle of $\sigma$ on which $\epsilon = 1$ we have
$\sigma_\epsilon^{-1}$ and $\epsilon \delta \sigma \delta
\sigma^{-1} \epsilon$ agree and on a cycle on which
$\epsilon = -1$ we have $\sigma_\epsilon$ and $\epsilon
\delta \sigma \delta \sigma^{-1} \epsilon$ agree.
\end{proof}

\begin{definition}\label{def:non_crossing}
Let $\pi \in S_n$ be a permutation of $[n]$ and $\gamma =
(1,2,3, \dots,\ab n)$. We say that $\pi$ is a
\textit{non-crossing permutation} if $\#(\pi) +
\#(\pi^{-1}\gamma) = n + 1$. We shall denote by $S_{NC}(n)$
the non-crossing permutations of $[n]$.
\end{definition}

\begin{remark}\label{remark:non-crossing_permutations}
We have already used the idea of taking a permutation of
$[n]$ and regard it as a partition of $[n]$ by using the
decomposition of the permutation into disjoint cycles and
making these the blocks of a partition. Biane \cite{biane}
showed that the permutations that satisfy $\#(\pi) +
\#(\pi^{-1}\gamma) = n + 1$ i.e. $g = 0$ in Equation
(\ref{eq:eulers_formula}), are exactly those whose cycles
form a non-crossing partition of $[n]$.
\end{remark}

\begin{proposition}\label{prop:sigma_non_crossing}
Let $\sigma \in S_n$ and $\epsilon \in \bZ_2^n$. Suppose
that $\epsilon$ is constant on the cycles of $\sigma$. Then
$f_\epsilon(\sigma) \leq 0$ with equality only if
$\sigma_\epsilon$ is a non-crossing permutation.
\end{proposition}

\begin{proof}
Let $\epsilon \delta \sigma \delta \sigma^{-1} \epsilon =
\delta \sigma_\epsilon \delta \sigma_\epsilon^{-1}$ as in
Lemma \ref{lemma:sigma_epsilon}.  As in the proof of Lemma
\ref{lemma:negativity_of_f} we have that
\begin{eqnarray*}
f_\epsilon(\sigma) & = &
{\textstyle\frac{1}{2}}
  \#(\gamma \delta \gamma^{-1} \delta \epsilon \delta
  \sigma\delta \sigma^{-1}\epsilon) + \#(\sigma) -(n+1) \\
& = &
{\textstyle\frac{1}{2}}
  \#(\gamma \delta \gamma^{-1} \delta \delta \sigma_\epsilon
  \delta \sigma^{-1}_\epsilon) + \#( \sigma_\epsilon) -(n + 1)
\\
& = &
{\textstyle\frac{1}{2}}
  \#( \sigma_\epsilon^{-1}\gamma \, \delta \gamma^{-1}
  \sigma_\epsilon \delta ) + \#(\sigma_\epsilon) - (n + 1) \\
& = &
\#(\sigma_\epsilon) + \#(\sigma_\epsilon^{-1}\gamma ) - (n + 1).
\end{eqnarray*}
By Equation (\ref{eq:eulers_formula}) we have
$f_\epsilon(\sigma) \leq 0$; and, according to Definition
\ref{def:non_crossing}, $\sigma_\epsilon$ is a non-crossing
permutation if and only if $f_\epsilon(\sigma) = 0$
\end{proof}

\begin{remark}\label{rem:epsilon_examples}
As an illustration let us consider two examples: $\epsilon
\equiv 1$ and $\epsilon \equiv -1$. First suppose $\epsilon
\equiv 1$, then $ \#(\epsilon
\gamma\delta\gamma^{-1}\epsilon \vee
\sigma\delta\sigma^{-1}) = {\textstyle\frac{1}{2}}
\#(\gamma\delta\gamma^{-1} \sigma\delta\sigma^{-1}) =
\#(\sigma^{-1}\gamma)$, so $\sigma_1 = \sigma$ and
$f_1(\sigma) = \#(\sigma) + \#(\sigma^{-1} \gamma) - (n+1) =
0$ only if $\sigma$ is non-crossing. When $\epsilon \equiv
-1$ we have that $ \#(\epsilon
\gamma\delta\gamma^{-1}\epsilon \vee
\sigma\delta\sigma^{-1}) = {\textstyle\frac{1}{2}}
\#(\delta\gamma\delta\gamma^{-1}\delta
\sigma\delta\sigma^{-1}) = \#(\sigma^{-1}\gamma)$, so
$\sigma_{_1} = \sigma^{-1}$ and $f_{-1}(\sigma) =
\#(\sigma^{-1}) + \#(\sigma \gamma) - (n+1) = 0$ only if
$\sigma^{-1}$ is non-crossing.
\end{remark}


\section{Limit Distributions}\label{sec:limit-distributions}

We assume that $d_1 d_2 \rightarrow \infty$ and that
$\ds\frac{p}{d_1d_2}\rightarrow c$, for some $0 < c < \infty$. Since
$W$ and $W^\str$ are Wishart matrices, their eigenvalue
distributions  converge to the
Marchenko-Pastur law with parameter $c$. Setting  $b = (1 +
\sqrt{c})^2$ and $a = (1 - \sqrt{c})^2$, this is the
distribution on $\bR^+$ that has density $\ds\frac{\sqrt{(b
    - t)(t - a)}}{2 \pi t}$ on $[a, b]$ and an atom of
weight $(1 - c)$ at $0$ if $ c < 1$.

The asymptotic eigenvalue distributions of $W^\ltr$ and
$W^\rtr$ were described by G. Auburn (see \cite{aubrun}) for
the case when $ d_1, d_2 \rightarrow \infty $, respectively
by T. Banica and I. Nechita for the case when $ d_1
\rightarrow \infty $ (see \cite{bn}). The calculations below
give another proof of these results and give some more
insight on the limit distributions.

\begin{lemma}\label{lemma:non-crossing-inverse}
Let $\sigma \in S_n$ and suppose that both $\sigma$ and
$\sigma^{-1}$ are non-crossing in the sense of Definition
\ref{def:non_crossing}. Then $\sigma$ can have only cycles
of size 1 or 2.
\end{lemma}

Before proving Lemma \ref{lemma:non-crossing-inverse} we
need to recall some standard facts about permutations and
pairings. We let $[\pm n] = \{1, -1, 2, -2, \dots, n,\ab
-n\}$. If $\sigma \in S_n$ is a permutation of $[n]$ then
$\sigma \delta \sigma^{-1}$ is a pairing of $[\pm n]$;
moreover if $(r, s)$ is a pair in this pairing then $r$ and
$s$ have opposite signs. We let $\cP_2^\delta(\pm n)$ be the
set of pairings of $[ \pm n]$ that only pair a positive
number to a negative number. There is a standard bijection
from $S_n$ to $\cP_2^\delta( \pm n)$ that we shall use. For
$\sigma \in S_n$ we have $\sigma \delta \sigma^{-1} \in
\cP_2^\delta( \pm n)$. If $\pi \in \cP_2^\delta(\pm n)$ then
$\pi \delta$ leaves $[n]$ invariant and so $\pi
\delta|_{[n]} \in S_n$. These two maps are inverses of each
other.

For example consider $\gamma = (1, 2, \dots, n) \in
S_n$. Then $\gamma \delta \gamma^{-1} = (-n, 1) \ab (-1,
2)(-2, 3) \cdots (-(n-1), n) \in \cP_2^\delta( \pm n)$ and
$(\gamma \delta \gamma^{-1}) \delta |_{[n]} = \gamma$.  Also
the permutation $\gamma\delta$ has the one cycle $(1, -1, 2,
-2, \dots, n, -n)$.

Inside $\cP_2^\delta(\pm n)$ we have the non-crossing
pairings of $[ \pm n]$ which only connect a positive number
to a negative number; we shall denote this subset by
$NC_2^\delta(\pm n)$.

\begin{lemma}\label{lemma:non-crossing_bijection}
The map $\sigma \mapsto \sigma \delta \sigma^{-1}$ is a
bijection from $S_{NC}(n)$ to $NC_2^\delta(\pm n)$.
\end{lemma}

\begin{proof}
We have to check that $\sigma \in S_{NC}(n)$ if and only if
$\sigma \delta \sigma^{-1} \in NC_2^\delta(\pm n)$. Note
that $\sigma \delta \sigma^{-1}$ is a pairing so that
$\#(\sigma \delta \sigma^{-1}) = n$. Also $\#( (\sigma
\delta \sigma^{-1})^{-1} \gamma \delta) = \#(\delta \sigma
\delta \sigma^{-1} \gamma) = \#(\delta \sigma \delta) +
\#(\sigma^{-1} \gamma) = \#(\sigma) + \#(\sigma^{-1}
\gamma)$ because $\delta \sigma \delta$ acts trivially on
$[n]$ and $\sigma^{-1}\gamma$ acts trivially on $[-n]$. Thus
$\#(\sigma \delta \sigma^{-1}) + \#( (\sigma \delta
\sigma^{-1})^{-1} \gamma \delta) = n + \#(\sigma) +
\#(\sigma^{-1} \gamma)$. By Remark
\ref{remark:non-crossing_permutations} we have that $\sigma$
is non-crossing if and only if $\sigma \delta \sigma^{-1}$
is non-crossing.
\end{proof}

\medskip\noindent
\textit{Proof of Lemma \ref{lemma:non-crossing-inverse}.} 
Suppose that $\sigma \in S_{NC}(n)$ and $i_1 < i_2 < i_3$
are distinct with $\sigma(i_1) = i_2$ and $\sigma(i_2) =
i_3$. Then $\gamma\delta$ visits $\{i_1, -i_1, i_2, -i_2,
i_3,\ab -i_3\}$ and $(i_1, -i_2)$ and $(i_2, -i_3)$ are
pairs of $\sigma^{-1} \delta \sigma$. Thus $\sigma^{-1}
\delta \sigma$ is not in $NC_2^\delta(\pm n)$ and hence by
Lemma \ref{lemma:non-crossing_bijection} $\sigma^{-1}
\not\in S_{NC}(n)$. Thus the only permutations $\sigma \in
S_{NC}(n)$ for which $\sigma^{-1} \in S_{NC}(n)$ are those
where $\sigma = \sigma^{-1}$, i.e. all cycles are singletons
or pairs.  
\hfill\qed

\begin{theorem}\label{thm:distrib}\cite[Thm. 1]{aubrun}
If $ d_1, d_2 \rightarrow \infty $, then the limit
distributions of $W^\ltr$ and $W^\rtr$ are semi-circular
with mean $c $ and variance $ c $.
\end{theorem}

\begin{proof}
We have just shown that the only non-vanishing cumulants of
the limiting distribution are $\kappa_1 = \kappa_2 =
c$. Thus the limiting distribution is semi-circular.
\end{proof}

\begin{remark}
The measure on $\bR$ whose free cumulants are $\kappa_1 =
\kappa_2 = c$ and $\kappa_n = 0$ for $n \geq 3$ is the
shifted semi-circle law. It has density $\frac{1}{2 \pi c}
\sqrt{ 4 c^2 - (t - c)^2}$ on the interval $[c - 2 \sqrt{c},
  c + 2 \sqrt{c}]$. We have used a different normalization
for $W$ than Aubrun, (we used $\frac{1}{d_1d_2}$ and he used
$\frac{1}{p}$), the advantage of ours is that the free
cumulants are very simple with this normalization.
\end{remark}

Next, we shall discuss the case when only one of the
parameters $ d_1 $, $ d_2$ approaches infinity and the other
one is fixed.

The following remarkable result is due to T.~Banica and
I.~Nechita \cite[Lemma 1.1]{bn}.
  
\begin{lemma}\label{lemma:bn}
Suppose that $ \sigma $ is a non-crossing permutation and
that $ \tau $ is a cycle of length $ n $ in $ S_n $. Then
\[
\# ( \sigma \tau ) = 1 + e ( \sigma)
\]
where $ e(\sigma ) $ is the number of cycles of $ \sigma $
of even length.
\end{lemma}
 
Let us recall the main result of \cite[Theorem 3.1]{bn},
which computes the free cumulants of the limit distribution
of $d_1 W^\ltr$ as $p/(d_1d_2) \rightarrow c$ but keeping
$d_1$ fixed.

\begin{theorem}\label{thm:distrib:2}
Suppose that $d_1$ is a fixed positive integer, and
$\frac{p}{d_1d_2} \longrightarrow c$ with $0 < c < \infty$.
The free cumulants of the limit distribution of $d_1 W^\ltr$ are
$\kappa_n = c d_1^{\,2}$ for $n$ even and $\kappa_n = c d_1$
for $n$ odd. This limit distribution is the free difference
of two Marchenko-Pastur laws one with parameter $cd_1
\frac{d_1 + 1}{2}$ and the other $c d_1 \frac{d_1-1}{2}$.
\end{theorem}

\begin{proof}
Let $\epsilon \equiv 1$ and $ \eta \equiv -1 $ in  Theorem \ref{thm:exp_mixed_mom}.
 By Remark
\ref{rem:epsilon_examples},
 $f_\epsilon(\sigma) < 0$ unless
$\sigma \in NC(n)$. For $\sigma \in NC(n)$ and $\eta
\equiv -1$ we have by Remark \ref{rem:epsilon_examples} and
Lemma \ref{lemma:bn}, $f_\eta(\sigma) - \#(\sigma) + n =
e(\sigma)$. Hence Theorem \ref{thm:exp_mixed_mom} gives
\[
\lim_{d_2 \rightarrow \infty}
\E(\tr \otimes \tr( (d_1 W^\ltr)^n ))
= \kern-0.75em
\sum_{\sigma \in NC(n)}
c^{\#(\sigma)} d_1^{f_{-1}(\sigma) + n}
= \kern-0.75em
\sum_{\sigma \in NC(n)}
(d_1 c)^{\#(\sigma)} d_1^{\,e(\sigma)}.
\]
Note that if we set $\kappa_n = d_1^2 c$ for $n$ even and
$\kappa_n = d_ 1 c$ for $n$ odd then $\kappa_\sigma = (d_1
c)^{\#(\sigma)} d_1^{e(\sigma)}$. This shows that the limit
distribution of $d_1 W^\ltr$ has the claimed
cumulants. Since $\kappa_n = (d_1c) \frac{d_1+1}{2} + (-1)^n
(d_1 c) \frac{d_1 - 1}{2}$, we have the claim
about the distribution being a free difference of
Marchenko-Pastur laws.
\end{proof}

\begin{remark}\label{remark:24}
If in Theorem \ref{thm:exp_mixed_mom} we let $ \epsilon \equiv -1 $ and $ \eta \equiv 1 $, the coefficients $d_1$ and $d_2$ switch roles, hence the argument above also gives an
analogous statement for holding $d_2$ fixed. 
More precisely, if $d_2$ is
fixed and $p/(d_1d_2) \rightarrow c$, then the free
cumulants of the limit distribution of $d_2 W^\rtr$ are
given by $\kappa_n = d_2^2 c$ for $n$ even and $\kappa_n =
d_2 c$ for $n$ odd. This distribution is also the free
difference of two Marchenko-Pastur distributions one of
parameter $d_2 c \frac{d_2+1}{2}$ and one of $d_2 c
\frac{d_2 - 1}{2}$.
\end{remark}

\begin{remark}
Since taking transposes preserves eigenvalue distributions
Theorem \ref{thm:distrib:2} and Remark \ref{remark:24} also gives us the free cumulants
of the limit distribution of $d_1 W^\rtr$ and $d_2W^\ltr$.
\end{remark}


\setbox1=\hbox{$d_1 W^\ltr$}

\section{A natural free decomposition of \box1{} when $d_1 = 2$}

\label{sec:natural-decomosition}

In \cite{bn} it was shown that the limit distribution of
$d_1W^\ltr$ can be written as the free difference of two
Marchenko-Pastur laws. The operators so obtained are not
related to the operator $d_1 W^\ltr$ though. In this section
we shall show that there is a natural decomposition of $d_1
W^\ltr$ when $d_1 = 2$, namely the diagonal and off diagonal
blocks, into free summands. More precisely, we let $w$ be
the limit distribution of $W$, which we can write $w$ as a
$2 \times 2$ matrix
\[
w = \frac{1}{2}
\begin{pmatrix} w_{11} & w_{12} \\ w_{21} & w_{22} \\ \end{pmatrix}.
\]
Relative to this block decomposition $2 W^\ltr$ converges to 
\[
2 w ^\ltr
=
\begin{pmatrix} w_{11} & w_{21} \\ w_{12} & w_{22} \\ \end{pmatrix}.
\]
We consider the two operators
\[
X_1 = \begin{pmatrix} w_{11} & 0 \\ 0 & w_{22} \end{pmatrix}
\mbox{\ and\ }
X_2 = \begin{pmatrix} 0 & w_{21} \\ w_{12} & 0 \end{pmatrix}
\]

The diagonal summand $X_1$ has the Marchenko-Pastur
distribution and the off diagonal summand $X_2$ is even and
has the same even cumulants as the diagonal summand. Our
main result in this section is that $X_1$ and $X_2$ are
free.

\begin{notation}
Let $d_1 = 2$, and suppose $p/(d_1d_2)$ converges to $c$
with $0 < c < \infty$.  Let $\{E_{11}, E_{12}, E_{21},
E_{22}\}$ be the standard matrix units in $M_2(\bC)$, but
viewed as elements of $M_2(\bC) \otimes M_{d_2}(\bC)$.
\end{notation}

\begin{lemma}\label{lemma:block_structure}
There is a $*$-non-commutative probability space $(\cA,
\phi)$ with elements $w, e_{11}, e_{12}, e_{21}, e_{22} \in
\cA$ such that $w$ has the Marchenko-Pastur distribution
with parameter $c$ and $\{e_{11}, e_{12}, e_{21}, e_{22}\}$
are matrix units in $\cA$ free from $w$. Moreover the joint
distribution of $\{W, E_{11}, E_{12},\ab E_{21}, E_{22}\}$
converges to that of $\{w, e_{11}, e_{12}, e_{21}, e_{22}\}$
\end{lemma}

\begin{proof}
As $W$, our Wishart
matrix, is unitarily invariant, it is asymptotically free
from our matrix units (see \cite[Theorem 4.9]{ms}). 
This is exactly the claim of the lemma.
\end{proof}

\begin{notation}
Thus we may 
write the matrix of $w$ with respect to the
matrix units $\{e_{11}, e_{12}, e_{21}, e_{22}\}$ as
\[
w = \frac{1}{2} 
\begin{pmatrix}
w_{11} & w_{12} \\
w_{21} & w_{22} \\
\end{pmatrix}.
\]
\end{notation}

We will let $\phi_1$ be the state on $e_{11}\cA e_{11}$
given by $\phi_1(x) = 2 \phi(x)$. The elements $\{w_{11},
w_{12}, w_{21}, w_{22}\}$ are in $e_{11}\cA e_{11}$ so their
cumulants must be computed relative to the state
$\phi_1$. When necessary we will denote these relative
cumulants by $\kappa_n^{(1)}$.

\begin{lemma}\label{lemma:w_11}
Each of $w_{11}$ and $w_{22}$
have the Marchenko-Pastur distribution with parameter
$d_1 c$.
\end{lemma}

\begin{proof}
By construction $w_{11} = e_{11}2we_{11}$. By \cite[Theorem 14.18]{ns}  
\begin{eqnarray*}
\kappa_n^{(1)}(w_{11}, \dots w_{11}) &=&
2^n \kappa_n(e_{11}we_{11}, \dots, e_{11}we_{11}) \\
&=&
2^{1} \kappa_n(w, \dots, w) = 2 c.
\end{eqnarray*}
\end{proof}

\begin{remark}\label{rem:entry_cumulants}
Elements $\{ a_{ij} \}_{i,j = 1}^n$ in a non-commutative
probability space $(\cA, \phi)$, they are called $R$-cyclic
if whenever $i_1, j_1, \dots, i_l, j_l \in [n]$ we have
$\kappa_l( a_{i_1j_1}, \dots, a_{i_lj_l}) = 0$ unless $j_1 =
i_2, \dots, j_{n-1} = i_n$ and $j_n = i_1$. By \cite[Example
  20.4]{ns} the elements $\{w_{11}, w_{12}, w_{21}, w_{22}
\}$ are $R$-cyclic. Moreover $\kappa_l(2w, \dots, 2w) = 2^l
\kappa_l(w, \dots, w) = 2^{l} c$.  So by \cite[Example
  20.4]{ns}, we have $\kappa_l^{(1)}(w_{i_1j_1}, w_{i_2j_2},
\dots, w_{i_lj_l}) = 2^{-l+1} \kappa_l(2w, \dots, 2w) = 2c$,
when $j_1 = i_2, \dots, j_{n-1} = i_n$ and $j_n = i_1$.
\end{remark}

Let $X_1 = \begin{pmatrix} w_{11} & 0 \\ 0 &
  w_{22}\\ \end{pmatrix}$ and $X_2 = \begin{pmatrix} 0 &
  w_{21} \\ w_{12} & 0\\ \end{pmatrix}$. Then 
$ 2 w^\ltr = X_1 + X_2$.   
 
\begin{lemma}\label{lemma:X-cumulants}
$X_1$ and $X_2$ are self-adjoint. The cumulants of $X_1$ are
  all equal to $2c$, i.e. $X_1$ is a Marchenko-Pastur
  operator with parameter $2c$. $X_2$ is an even operator in
  that it is self-adjoint and all of its odd moments are
  $0$. The even cumulants of $X_2$ are all equal to $2c$.
\end{lemma}

\begin{proof}
We have $\phi(X_1^l) = \phi^{(1)}(w_{11}^l)$ so $X_1$ and
$w_{11}$ have the same cumulants, which by Remark
\ref{rem:entry_cumulants} are all $2c$. Because $X_2$ is off
diagonal and self-adjoint, it is an even operator. By
\cite[Proposition 15.12]{ns} the cumulants of $X_2$ are the
$*$-cumulants of $w_{21}$. In Remark
\ref{rem:entry_cumulants} we observed that these are all
$2c$.
\end{proof}

Our next goal is to show that $X_1$ and $X_2$ are free in
$(\cA, \phi)$. This is somewhat surprising in that $X_1$ and
$X_2^\ltr$ are not free.  By $X_2^\ltr$ we mean the matrix
$\begin{pmatrix} 0 & w_{12} \\ w_{21} & 0 \end{pmatrix}$.
To see this note that $\phi(X_1X_2^\ltr X_2^\ltr X_1) = 2c +
3 (2c)^2 + (2c)^3$ whereas if $X_1$ and $X_2$ were free we would
have $\phi(X_1 X_2^\ltr X_2^\ltr X_1) = (2c)^2 + (2c)^3$.
This gives another unexpected instance where a partial
transpose produces freeness, but this time at the level of
operators.

Now let us turn to the freeness of $X_1$ and $X_2$. Let 
\[
Y_1 = X_1, \ 
Y_2 = 
\begin{pmatrix}
w_{21}, &   0    \\
0       &  w_{12}
\end{pmatrix} 
\mbox{\ and\ }
Y_3 = 
\begin{pmatrix}
0   &   1    \\
1       &  0
\end{pmatrix}.
\]
Then $X_2 = Y_2 Y_3$. Let $i_1, \dots, i_n$ be such that
$\ker(i) < 1_n$. Let $r$ be the number of times $i = 2$, and
$q = n + r$. Then there are $j_1, j_2, \dots, j_q \in \{1,
2, 3\}$ such that
\[
X_{i_1} \dots X_{i_n} 
=
Y_{j_1} \cdots Y_{j_q}.
\]
We now apply the formula for cumulants with products as
entries \cite[Theorem 11.12]{ns}.  Then
\[
\kappa_n(X_{i_1}, \dots, X_{i_n})
=
\sum_{\pi \in NC(q)}
\kappa_\pi (Y_{j_1}, \dots, Y_{j_q})
\]
where the sum runs over all non-crossing partitions in
$NC(q)$ such that $\pi \vee \rho = 1_q$ and $\rho$ is the
non-crossing partition whose blocks are have either 1 or 2
elements, and the singletons are where $j_l = 1$ appears in
the string $Y_{j_1} \cdots Y_{j_q}$, and the pairs are $(l,
l+1)$ where $j_l = 2$ and $j_{l+1} = 3$. Since $\ker(i) <
1_n$ there must both singletons and pairs. Let us consider a
$\pi \in NC(q)$ which is such that $\pi \vee \rho = 1_q$ and
we shall show that by the $R$-cyclicity of $w$ we have
$\kappa_\pi (Y_{j_1}, \dots, Y_{j_q}) = 0$. Summing over all
such $\pi$ we get that $\kappa_n(X_{i_1}, \dots, X_{i_n}) =
0$. This shows that all mixed cumulants vanish and hence
that $X_1$ and $X_2$ are free as claimed.

\begin{lemma}\label{lemma:even-Y-3}
Given $p_1, \dots , p_s
\in[3]$ we have $\phi(Y_{p_1} \cdots Y_{p_s})= 0$ unless
$Y_3$ appears an even number of times.
\end{lemma}

\begin{proof}
$Y_1$ and $Y_2$ are diagonal so $Y_{p_1} \cdots Y_{p_s}$
will be 0 on the diagonal unless $Y_3$ appears an even
number of times.  
\end{proof}

\begin{lemma}\label{lemma:even-Y-2-cumulant}
Given $p_1, \dots , p_s \in[3]$ we have $\kappa_s(Y_{p_1},
\dots, Y_{p_s})= 0$ unless $Y_3$ appears an even number of
times.
\end{lemma}

\begin{proof}
We write
\[
\kappa_s(Y_{p_1}, \dots, Y_{p_s})= 
\sum_{\pi \in NC(s)}
\mu(\pi, 1_s) \phi_\pi(Y_{p_1}, \dots, Y_{p_s}).
\]
Given $\pi$, we have by Lemma \ref{lemma:even-Y-3}, that
each block of $\pi$ must contain an even number of $Y_3$'s,
or else $\phi_\pi(Y_{p_1}, \dots, Y_{p_s}) = 0$. Summing
over all blocks of $\pi$ we get that the number of $Y_3$'s
is even.
\end{proof}

\begin{definition}
Let $i_1, \dots, i_s \in [3]$ we say that the $s$-tuple has
the property (\textit{nvc}) if each non-zero entry of
$Y_{i_1} \cdots Y_{i_s}$ is of the form
\[
w_{u_1v_1} \cdots w_{u_kv_k}
\]
where $v_1 = u_2, v_2 = u_3, \dots, v_{k-1} = v_k$. Note
that we do not require $v_k = u_1$ as in $R$-cyclicity. We
say that the string has property (\textit{vc}) if it does
not have property (\textit{ncv}).
\end{definition}

\begin{remark}
We now describe the generic sequences with property
(\textit{nvc}). First we have any power of $Y_1$. The
product of two or more $Y_2$'s does not have property
(\textit{nvc}). No power of $Y_3$ has property
(\textit{nvc}), because all its entries are either 0 or 1.

Now suppose we start with a $Y_2$. We can only follow with a
$Y_1$ or a $Y_3$. So our basic reduced sequence is (with the
possibility that $k =0$)
\[
Y_2 \underbrace{Y_1 \cdots Y_1}_k Y_3 Y_2. 
\]
We can enhance this by putting an even power of $Y_3$
between any two letters above. Note that there cannot be an
odd number of $Y_3$'s between two $Y_1$'s as
\[
Y_1 Y_3 Y_1 =
\begin{pmatrix}
     0       &  w_{11}w_{22} \\
w_{22}w_{11} &      0        \\
\end{pmatrix}
\]
So the most general string starting and ending with a $Y_2$ is
\[
Y_2 Y_3^{l_1} Y_1^{k_1} Y_3^{l_2} Y_1^{k_2} \cdots
Y_1^{k_r} Y_3^{l_{r+1}} Y_2
\]
with $l_1, \dots, l_r$ even and $l_{r+1}$ odd.
\end{remark}

\begin{lemma}
Let $i_1, \dots, i_k$ be a string with property
(\textit{nvc}) which starts and ends with $Y_2$ and has no
other $Y_2$'s. Then the number of $Y_3$'s is odd.
\end{lemma}

\begin{proof}
We just observed that the number of $Y_3$'s is $l_1 + \cdots
+ l_r + l_{r+1}$ which is odd.
\end{proof}

\begin{lemma}\label{lemma:balanced-blocks}
If $\pi \in NC(q)$ and $\pi \vee \rho = 1_q$ and
$\kappa_\pi(Y_{i_1}, \dots, Y_{i_q}) \not = 0$ then each
block of $\pi$ must contain the same number of $Y_2$'s as
$Y_3$, and both numbers are even.
\end{lemma}

\medskip\noindent
\begin{proof}
We have just observed that the number of $Y_3$'s between
$Y_2$'s is odd. Thus to go all the way round a cycle the
number of $Y_3$'s is equal to the number of $Y_2$ plus an
even number which might be 0. However in our whole string
the number of $Y_2$'s and $Y_3$'s is the same. If one cycle
had an excess of $Y_3$'s then another cycle would have a
deficit. Thus all cycles must be balanced. Since we already
know that each cycle has an even number of $Y_3$'s it also
has an equal even number of $Y_2$'s.
\end{proof}

\begin{lemma}
Let $i_1\, \dots, i_q \in [3]$ be such that $\ker(i) < 1_q$
and $\pi \in NC(q)$ be such that $\pi \vee \rho = 1_q$. Then
$\kappa_\pi(Y_{i_1}, \dots, Y_{i_k}) = 0$.
\end{lemma}

\begin{proof}
Let $V$ be a block of $\pi$ that contains a $l$ such that
$j_l = 1$. Then $(l)$ is a block of $\rho$. Since we are
assuming that $\pi \vee \rho = 1_q$ there must be $l_0 \in
V$ with $j_{l_0} \in \{2, 3\}$.  If the contribution of this
block to $\kappa_\pi(Y_{i_1}, \dots, Y_{i_k})$ is not 0 then
there must be a $Y_1$ followed by a $Y_3$. So we may assume
that we have a $l$ and $l'$ such that $j_l =1$, $j_{l'} = 3$
and $l'$ follows $l$ in $V$. We have that $\pi$ restricts to
a non-crossing partition of $[l+1, l'-1]$. Each block in
this restriction contains the same number of $Y_2$'s as
$Y_3$'s by Lemma \ref{lemma:balanced-blocks}. However this
impossible because in the original string $Y_{i_1}, \dots,
Y_{i_k}$ a $Y_2$ is always followed by a $Y_3$ and we have
removed one $Y_3$. Thus $\kappa_\pi(Y_{i_1}, \dots, Y_{i_k})
= 0$.
\end{proof}

\begin{theorem}\label{thm:freeness-X-1-and-X-2}
$X_1$ and $X_2$ are free in $(\cA, \phi)$. 
\end{theorem}

\begin{proof}
We have just shown that by the formula for cumulants with
products for entries we have that mixed cumulants
vanish. Thus $X_1$ and $X_2$ are free.
\end{proof}

\begin{remark}
The distribution of $w^\ltr$ in $(\cA, \phi)$ is the limit
distribution of $W^\ltr$ which is the same as $W^\rtr$. Thus
the distribution of $d_1w^\rtr$ is the same as that of
$d_1w^\ltr$.
\end{remark} 
 
\begin{theorem}\label{thm:free-additive-version}
For $d_1 = 2$ and $p/(d_1d_2) \rightarrow c$ the
limit distribution of $2W^\ltr$ is the free additive
convolution of a Marchenko-Pastur law with parameter $2c$ and
an even operator with all even cumulants equal to $2c$.
\end{theorem}


\section{Asymptotic Freeness}\label{sec:asymptotic-freeness}

Since $ W $ is unitarily invariant, a consequence of the
results from \cite{mp2} is that $ W $ and $ W^T $ are
asymptotically free if $ d_1d_2 \rightarrow \infty $. In
this section we will present the main results of the paper,
which, using the relation form Theorem
\ref{thm:exp_mixed_mom}, gives an improvement of the result
mentioned above.
 
 \begin{theorem}\label{thm:main}
 If $ d_1\rightarrow \infty $ and  
   $ d_2 \rightarrow \infty  ,$
    then the family 
    $\{ W, W^T, W^\Gamma,\ab W^\ltr \} $ is asymptotically free.
 \end{theorem}

\begin{proof}

By Theorem \ref{thm:exp_mixed_mom} we have that
\[
\E(\tr \otimes \tr(W^{(\epsilon_1, \eta_1)} \cdots
W^{(\epsilon_n, \eta_n)}))
= 
\sum_{\sigma\in S_n}
\bigg(\frac{p}{d_1d_2}\bigg)^{\#(\sigma)}
d_1^{\ f_\epsilon(\sigma)}
d_2^{\ f_\eta(\sigma)}
\]
and by Lemma \ref{lemma:negativity_of_f} and Proposition
\ref{prop:sigma_non_crossing} we have that
\begin{itemize}

\item
$f_\epsilon(\sigma), f_\eta(\sigma) \leq 0 $ for all
  $\sigma, \epsilon$, and $\eta$;

\item
$f_\epsilon(\sigma), f_\eta(\sigma) <0 $ unless $\epsilon$
  and $\eta$ are constant on the cycles of $\sigma$;

\item
$f_\epsilon(\sigma) <0 $ unless $\sigma_\epsilon$ is
  non-crossing.
\end{itemize}
Thus when $d_1, d_2 \rightarrow \infty$ and $ \displaystyle
 \frac{p}{d_1 d_2} \rightarrow c$ 
 we need only consider $\sigma$'s for
which
\begin{enumerate}

\item
$\epsilon$ and $\eta$ are constant on the cycles of $\sigma$;

\item
both $\sigma_\epsilon$ and $\sigma_\eta$ are non-crossing.

\end{enumerate}

Note that as
partitions $\sigma$, $\sigma_\epsilon$, and $\sigma_\eta$
are the same, as the only possible difference between them
is whether we reverse the order of elements in a cycle of
$\sigma$. Thus we have shown that the limit when $d_1, d_2 \rightarrow \infty $ of an arbitrary mixed moment
can be written as a sum over non-crossing partitions; that
means that the terms that appear are the free cumulants of
the mixed moment we are considering. However, by
(\textit{i}), the blocks of $\sigma$ only connect
$W^{(\epsilon_i, \eta_i)}$ to $W^{(\epsilon_j,\eta_j)}$ if
$(\epsilon_i, \eta_i) = (\epsilon_j, \eta_j)$. This means we
have shown that mixed cumulants vanish and this implies
the conclusion.

\end{proof}

\begin{theorem}
\emph{(\textit{i})}
If $d_1 \rightarrow \infty $ and $ d_2 $ is fixed, then the
family $\{ W, W^\Gamma \} $ is asymptotically free from the
family $\{ W^T, W^\ltr\} $, but $ W $ is \emph{not}
asymptotically free from $ W^\Gamma $, nor is $ W^T $ from
$W^\ltr $.

\smallskip\noindent (\textit{ii}) If $ d_1 $ is fixed and
$d_2 \rightarrow \infty $, then the family $ \{ W, W^\ltr \}
$ is asymptotically free from the family $\{ W^T, W^\Gamma
\} $ but $ W $ is \emph{not} asymptotically free from $
W^\Gamma $, nor is $ W^T $ from $W^\Gamma $.
\end{theorem}

\begin{proof}
Suppose first that $d_1 \rightarrow \infty $ and $ d_2 $ is
fixed. Hence in the summation from Theorem
\ref{thm:exp_mixed_mom} only the terms where
$f_{\varepsilon} (\sigma) = 0 $ will contribute to the
limit. As in the proof of Theorem \ref{thm:main}, the last
condition is equivalent to $ \sigma_{\varepsilon} $ is
non-crossing and $ \varepsilon $ is constant on the cycles
of $\sigma $. Since the partitions $\sigma $ and
$\sigma_{\varepsilon} $ are the same, it follows that the
limit as $ d_1 \rightarrow \infty $ of an arbitrary mixed
moment is written as a sum over non-crossing partitions, so
the terms in the right-hand side are in fact free
cumulants. The condition that $ \varepsilon $ is constant on
the cycles of $ \sigma $ gives that only free cumulants in
elements from only one of the families from part
(\textit{i}) do not vanish, hence the asymptotic freeness is
proved.
 
For the second part of (\textit{i}), we will use the
expansion of $ E \circ \tr \otimes \tr (W \cdot W^\Gamma ) $
from Theorem \ref{thm:exp_mixed_mom} in the case $
\varepsilon=(1, 1) $ and $ \eta = (1, -1)$.  Note that $ S_2
$ contains only two permutations, $ \gamma = (1, 2) $ and $
\sigma = (1), (2) $, both non-crossing. Also, since $
\varepsilon $ is constant, it is constant on the cycles of $
\sigma$ and $ \gamma$.  It follows that $
f_{\varepsilon}(\sigma) = f_{\varepsilon}(\gamma) =0.$
Moreover, $ \eta $ is constant on the cycles of $ \sigma $
and $ \sigma_{ \eta} = \sigma $ is non-crossing, hence $
f_\eta (\sigma) = 0 $. Therefore, Theorem 8 gives that
\[ 
E \circ \tr \otimes \tr \big( W \cdot W^\Gamma \big) = c^2 +
c\cdot d_2^{ f_{\eta}(\gamma)}
\]
As $ d_1 \rightarrow \infty $, the first moment of $ W $
approaches $ c $, and from Theorem \ref{thm:distrib:2}, so
does the first moment of $W^{ \Gamma } $, hence
   \[
   \lim_{ d_1 \rightarrow \infty}
   \kappa_2( W, W^{ \Gamma }) = 
   c\cdot d_2^{ f_{\eta}(\gamma)} \neq 0.
   \]
The same argument also shows that $ W^T $ and $ W^\ltr $ are
not asymptotically free, since $ W^\ltr = \big( W^{ \Gamma
}\big)^T $.
   
Finally, part (\textit{ii}) also follows from the argument
for part (\textit{i}), since the relation from Theorem
\ref{thm:exp_mixed_mom} is symmetric in $ ( d_1,
\varepsilon) $ and $ ( d_2, \eta) $.

\end{proof}


\section{The Case of Real Wishart Matrices}\label{sec:real-case}

In this section we examine the case of real Wishart
matrices. More precisely, $ W $ will denote now the
symmetric $ d_1 d_2 \times d_1 d_2 $ random matrix
\[
W = \frac{1}{ d_1 d_2 } \big( G_i G_j^\ast \big)_{ i, j
  =1}^{ d_1 }.
\]
where $ \{ G_i : 1 \leq i \leq d_1 \} $ is a family of $ d_2
\times p $ random matrices whose entries are independent
Gaussian random variables of mean 0 and variance 1.

Since $W= W^\str$ and $W^\ltr = W^\rtr$ we shall
only work with $W$ and $W^\rtr$. For this reason we shall
use slightly different notation that in the previous
sections. For $\epsilon \in \bZ_2 = \{-1, 1\}$ we let
\[
W^{(\epsilon)} =
\begin{cases} W & \epsilon = 1\\ W^\rtr & \epsilon = -1
\end{cases}.
\]
Thus our goal will be to consider, $W^{(\epsilon_1)} \cdots
W^{(\epsilon_n)}$, an arbitrary word in $W$ and $W^\rtr$ and
find its limiting expectation.

\begin{theorem}\label{thm:24}
With the notations from above, we have that
\[
\E(\tr \otimes \tr (W^{(\epsilon_1)} \cdots W^{(\epsilon_n)})) 
=
\sum_{\pi \in \cP_2(\pm n)} 
\Big( \frac{p}{d_1 d_2}\Big)^{\#(\pi\delta)/2}
d_1^{g(\pi)} d_2^{g(\epsilon \pi \epsilon)}
\]
where 
\[
g(\pi) = \#(\gamma \delta \gamma^{-1} \vee \pi) +
\#(\pi\delta)/2 - (n + 1)
\] 
and $ \epsilon \in S( \pm n ) $ is, as before, given by 
\[
\epsilon (k) = \left\{ 
\begin{array}{cl}
k, & \text{if} \  \epsilon_{ | k | } = 1 \\
- k & \text{if} \  \epsilon_{ | k | } = - 1.
\end{array}
\right.
\] 
\end{theorem}

\begin{proof}
\begin{eqnarray*}\lefteqn{%
\Tr \otimes \Tr (W^{(\epsilon_1)} \cdots W^{(\epsilon_n)}) }\\
& = &
\mathop{\sum_{i_{\pm 1}, \dots, i_{\pm n} = 1}^{d_1}}_
{i = i \circ \gamma \delta \gamma^{-1}}
\Tr ( W(i_1, i_{-1})^{(\epsilon_1)} \cdots W(i_n
i_{-n})^{(\epsilon_n)}) \\
& = & 
\sum_{i_{\pm 1}, \dots, i_{\pm n}}
\mathop{\sum_{k_{\pm 1}, \dots, k_{\pm n} = 1}^{d_2}}_
{k = k \circ \gamma \delta \gamma^{-1}}
( W(i_1, i_{-1})^{(\epsilon_1)})_{k_1 k_{-1}} \cdots (W(i_n
i_{-n})^{(\epsilon_n)})_{k_n k_{-n}} \\
& = & 
\sum_{i_{\pm 1}, \dots, i_{\pm n}}
\mathop{\sum_{j_{\pm 1}, \dots, j_{\pm n} = 1}^{d_2}}_
{j = j \circ \epsilon \gamma \delta \gamma^{-1} \epsilon}
W(i_1, i_{-1})_{j_1 j_{-1}} \cdots W(i_n i_{-n})_{j_n j_{-n}} \\
& = &
(d_1d_2)^{-n}
\mathop{\sum_{i_{\pm 1}, \dots, i_{\pm n}}}_%
{j_{\pm 1}, \dots, j_{\pm n}} 
\sum_{k_1, \dots, k_n=1}^p
g^{(i_1)}_{j_1k_1} g^{(i_{-1})}_{j_{-1}k_1} \cdots
g^{(i_n)}_{j_nk_n} g^{(i_{-n})}_{j_{-n}k_n}. \\
\end{eqnarray*}
In line 3 we momentarily break with our previous convention
about $W^{(\epsilon)}$ indicating whether or not we take a
partial transpose; in this case $W(i_u,i_{-u})^{(-1)}$ means
take the transpose of the $d_2 \times d_2$ matrix
$W(i_u,i_{-u})$. In passing from line 3 to line 4 above we
let $j = k \circ \epsilon$.

Now 
\[
\E(g^{(i_1)}_{j_1k_1} g^{(i_{-1})}_{j_{-1}k_1} \cdots
g^{(i_n)}_{j_nk_n} g^{(i_{-n})}_{j_{-n}k_n})
=
\sum_{\pi \in \cP_2(\pm n)}
\prod_{(r, s) \in \pi} \E(g^{(i_r)}_{j_rk_r} g^{(i_s)}_{j_sk_s}).
\]
On the right hand side in the expression above we are
extending $k$ as a function from $[n]$ to $[p]$ to a
function on $[\pm n]$ by requiring $k_r = k_{-r}$. This
means $k = k \circ \delta$. Now $\E(g^{(i_r)}_{j_rk_r}
g^{(i_s)}_{j_sk_s}) = 0$ unless $i_r = i_s$, $j_r = j_s$,
and $k_r = k_s$, in which case it is 1. Thus
\begin{eqnarray*}\lefteqn{%
\E(\Tr \otimes \Tr (W^{(\epsilon_1)} \cdots W^{(\epsilon_n)})) }\\
& = &
\sum_{\pi \in \cP_2(\pm n)}
d_1^{\#(\gamma \delta \gamma^{-1} \vee \pi) -n}
d_2^{\#(\epsilon \gamma \delta \gamma^{-1}\epsilon \vee \pi) -n}
p^{\#(\pi\delta)/2}.
\end{eqnarray*}
Since we require $k = k \circ \pi$ and $k = k \circ \delta$
we must have $k = k \circ \pi\delta$. Now as noted in Lemma
\ref{lemma:pairingproduct} the cycles of $\pi\delta$ appear
in pairs where one part of a pair is the conjugate by
$\delta$ of the other. Since $k$ is a function on $[n]$,
$\#(\pi\delta)$ double counts the degrees of freedom. Hence
the exponent of $p$ is $\#(\pi\delta)/2$. Thus
\begin{eqnarray*}\lefteqn{%
\E(\tr \otimes \tr (W^{(\epsilon_1)} \cdots
W^{(\epsilon_n)}))}\\ & = & \kern-1.0em \sum_{\pi \in
    \cP_2(\pm n)} \kern-0.5em \Big( \frac{p}{d_1
    d_2}\Big)^{\#(\pi\delta)/2} d_1^{\#(\gamma \delta
    \gamma^{-1} \vee \pi) + \#(\pi\delta)/2 - (n + 1)} \\ &
  & \qquad\qquad\qquad\quad\mbox{}\times d_2^{\#(\epsilon
    \gamma \delta \gamma^{-1}\epsilon \vee \pi) +
    \#(\pi\delta)/2 - (n + 1)}.
\end{eqnarray*}
Finally, note that 
\begin{eqnarray*}\lefteqn{
\#(\epsilon
\gamma \delta \gamma^{-1}\epsilon \vee \pi) +
\#(\pi\delta)/2 - (n + 1) } \\
& = & \frac{1}{2}\#(\gamma \delta
\gamma^{-1} \epsilon \pi \epsilon) + \#(\epsilon \pi \delta
\epsilon)/2 - (n + 1) = g(\epsilon \pi \epsilon),
\end{eqnarray*}
hence the conclusion.
\end{proof}

Next we shall show that $g(\pi) \leq 0$ and
$g( \epsilon \pi \epsilon ) \leq 0$ for all $\epsilon$ and $\pi$, and
to find for which pairings $\pi$ we have equality.

\begin{lemma}
Let $\pi \in \cP_2(\pm n)$ be a pairing such that there is
$(r, s) \in \pi$ with the same sign. Then $g(\pi) < 0$.
\end{lemma}

\begin{proof}
Since $\pi$ connects two elements with the same sign,
$\pi\delta$ connects two elements with opposite signs. Then
the subgroup generated by $\gamma \delta \gamma^{-1} \delta$
and $\pi\delta$ acts transitively on $[\pm n]$.  Thus
\[
\#(\pi\delta) +
\#((\pi\delta)^{-1}\gamma \delta \gamma^{-1} \delta)
+ \#(\gamma \delta \gamma^{-1} \delta) \leq 2(n+1).
\]
We have 
\[
2\#(\gamma \delta \gamma^{-1} \vee \pi ) = \#(\gamma \delta
\gamma^{-1} \pi ) = \#(\gamma \delta \gamma^{-1} \delta
\delta \pi) = \#( (\pi \delta)^{-1} \gamma \delta
\gamma^{-1} \delta).
\]
Thus $g(\pi) = \#(\gamma \delta \gamma^{-1} \vee \pi ) +
\#(\pi\delta)/2 - (n + 1 ) \leq -1$.
\end{proof}

\begin{lemma}\label{lemma:pi_delta_invariance}
Suppose $\pi \in \cP_2(\pm n)$ and $\pi$ only connects
elements of opposite sign. Then $\pi \delta$ leaves $[n]$
invariant and $g(\pi) \leq 0$ with equality only if $\pi
\delta|_{[n]}$ is a non-crossing permutation.
\end{lemma}

\begin{proof}
Since both $\pi$ and $\delta$ switch signs, $\pi \delta$
preserves signs. Thus $\pi \delta$ leaves $[n]$
invariant. By Lemma \ref{lemma:pairingproduct} we have
$\#(\pi \delta) = 2 \#(\pi \delta|_{[n]})$. Also $2\#(\gamma
\delta \gamma^{-1} \vee \pi) = \#(\gamma \delta \gamma^{-1}
\delta \delta \pi) = \#( (\pi \delta)^{-1} \gamma \delta
\gamma^{-1} \delta) = 2 \#((\pi\delta|_{[n]})^{-1}
\gamma)$. Hence
\[
g(\pi) = \#(\gamma \delta \gamma^{-1} \vee \pi) + \#(\pi
\delta)/2 - (n + 1) \leq 0
\]
with equality only if $\pi\delta|_{[n]}$ is a non-crossing
permutation.
\end{proof}

\begin{lemma}\label{lemma:27}
Let $\epsilon \in \bZ_2^n$ and $\pi \in \cP_2(\pm n)$. Then
$g(\epsilon\pi\epsilon) < 0$ unless $\epsilon \pi \delta \epsilon$
leaves $[n]$ invariant. If $\epsilon \pi \delta \epsilon$
leaves $[n]$ invariant then $g(\epsilon\pi \epsilon) \leq 0$ with
equality only if $\epsilon \pi \delta \epsilon|_{[n]}$ is a
non-crossing permutation.
\end{lemma}

\begin{proof}
By
Lemma \ref{lemma:pi_delta_invariance} we have
$g(\epsilon\pi \epsilon) < 0$ unless $\epsilon \pi \epsilon \delta =
\epsilon \pi \delta \epsilon$ leaves $[n]$ invariant. If
$\epsilon \pi \delta \epsilon$ leaves $[n]$ invariant then again
by Lemma \ref{lemma:pi_delta_invariance} we have
$g(\epsilon\pi \epsilon ) \leq 0$ with equality only if $\epsilon\pi
\delta \epsilon|_{[n]}$ is a non-crossing permutation.
\end{proof}

\begin{lemma}\label{lemma:28}
Let $\epsilon \in \bZ_2^n$ and $\pi \in \cP_2(\pm
n)$. Suppose $\pi\delta$ leaves $[n]$ invariant. Then
$\epsilon\pi \delta \epsilon$ leaves $[n]$ invariant if and
only if $\epsilon$ is constant on the cycles of $\pi
\delta$.
\end{lemma}

\begin{proof}
Suppose $(i_1, \dots, i_k)$ is a cycle of $\pi \delta$. All
these elements must have the same sign. The the
corresponding cycle of $\epsilon \pi \delta \epsilon$ is
$(\epsilon(i_1), \dots, \ab\epsilon(i_k))$. The elements of
$\epsilon \pi \delta \epsilon$ is $(\epsilon(i_1), \dots,
\ab\epsilon(i_k))$ have the same sign if and only if
$\epsilon$ is constant on $\epsilon \pi \delta \epsilon$ is
$(\epsilon(i_1), \dots, \ab\epsilon(i_k))$.
\end{proof}

The following theorem is the main result of this
section. Recall from Lemma \ref{lemma:sigma_epsilon} that if
$\epsilon$ is constant on the cycles of $\sigma$, then we
obtain $\sigma_\epsilon$ from $\sigma$ by reversing the
cycles on which $\epsilon = -1$.

\begin{theorem}\label{thm:main_real}
\[
\lim_{d_1, d_2 \rightarrow \infty}
\E(\tr \otimes \tr (W^{(\epsilon_1)} \cdots W^{(\epsilon_n)}))
=
\sum_{\sigma \in S_{NC}(n)} c^{\#(\sigma)}
\]
where the sum runs over all non-crossing permutations
$\sigma$ such that $\epsilon$ is constant on the cycles of
$\sigma$ and $\sigma_\epsilon$ is also non-crossing.
\end{theorem}

\begin{proof}
In the formula from Theorem \ref{thm:24}, only the pairings
$ \pi $ such that $ g(\pi) = g( \epsilon \pi \epsilon) = 0 $
will contribute to the summation when $ d_1, d_2 \rightarrow
\infty $.

Recall that $\cP_2^\delta(\pm n)$ denotes the pairings $\pi$
of $[ \pm n]$ such that $\pi\delta$ leaves $ [n]$
invariant. For such a $\pi$ we let $\sigma = \pi \delta|
_{[n]}$ be the corresponding permutation. We already noted
that this is a bijection from $\cP_2^\delta( \pm n)$ to
$S_n$ and $\pi \delta = \delta \sigma^{-1} \delta
\sigma$. From Lemma \ref{lemma:pi_delta_invariance}, the
condition $ g(\pi) = 0 $ implies that $ \sigma =
\delta\sigma^{-1} \delta \sigma_{ \ [ n ] } $ is
noncrossing.

According to Lemmas \ref{lemma:27} and \ref{lemma:28}, the
condition $ g( \epsilon \pi \epsilon) = 0 $ implies that $
\epsilon $ is constant on the cycles of $ \sigma $.
As in Lemma \ref{lemma:sigma_epsilon},
$\epsilon \pi \delta \epsilon = \delta \sigma_ \epsilon^{-1}
\delta \sigma_\epsilon$. Therefore
\begin{align*}
\#( \epsilon \gamma \delta \gamma^{-1} \epsilon \vee \pi)
&= 
\frac{1}{2}\#(\gamma \delta \gamma^{-1} \delta (\epsilon \pi \delta
\epsilon)^{-1})\\
& =
\frac{1}{2} \#( (\delta \sigma_\epsilon^{-1} \delta
 \sigma_\epsilon)^{-1} \gamma \delta \gamma^{-1} \delta)
  = 
 \#(\sigma_\epsilon ^{-1} \gamma),
 \end{align*}
 which gives
 \begin{equation*}
 g(\epsilon \pi \epsilon) = \#(\sigma_\epsilon) + \#(\sigma_\epsilon^{-1}\gamma) - (n
 + 1).
 \end{equation*}
hence the formula (\ref{eq:eulers_formula}) gives that $
g(\epsilon\pi\epsilon) \leq 0 $ with equality if and only if
$ \sigma_{\epsilon} $ is non-crossing.
  
\end{proof}

An immediate consequence of Theorem \ref{thm:main_real} is
part (1) of the following result.

\begin{theorem}\label{distrib-real}
Suppose $p/(d_1d_2) \rightarrow c$.

\smallskip\noindent
\emph{(1)} If $ d_1, d_2 \rightarrow \infty $, then $W^\rtr$
is asymptotically a shifted semi-circular operator with
$\kappa_1 = \kappa_2 = c$.

\smallskip\noindent
\emph{(2)} If $d_1 \rightarrow \infty$ and $ d_2 \geq 2$ is
fixed then the asymptotic distribution of $ d_2 W^\rtr$,
equals the distribution of the difference of two free
variables with Marchenko-Pastur laws, the first of parameter
$\displaystyle cd_2\frac{d_2 + 1}{2} $ and the second of
parameter $\displaystyle cd_2\frac{d_2 - 1}{2} $.

\smallskip\noindent
\emph{(3)} If $ d_1 $ is fixed and $ d_2 \rightarrow \infty
$ then the asymptotic distribution of $ d_1 W^\rtr $, equals
the distribution of the difference of two free variables
with Marchenko-Pastur laws, the first of parameter
$\displaystyle cd_1\frac{d_1 + 1}{2} $ and the second of
parameter $\displaystyle cd_1\frac{d_1 - 1}{2} $.
    
\end{theorem}

  \begin{proof}
Letting  $ \epsilon _j = -1 $ for all $ j =1, \dots, n $ in Theorem \ref{thm:24}, we obtain that
  \begin{equation}\label{p-t:moments}
  E \circ \tr \otimes \tr \big( (W^\Gamma)^n  \big) 
  = \sum_{\pi \in \cP_2(\pm n)} 
  \Big( \frac{p}{d_1 d_2}\Big)^{\#(\pi\delta)/2}
  d_1^{g(\pi)} d_2^{g(\delta \pi \delta)} .
  \end{equation}  
  
   Suppose first that $ d_1, d_2 \rightarrow \infty $. Then, in the summation from (\ref{p-t:moments}), only terms with $ g(\pi) = g ( \delta \pi \delta) = 0 $ will contribute to the limit. From Theorem \ref{thm:main_real}, this is equivalent to both $ \sigma $ and $ \sigma_\delta $ be noncrossing. But
    $ \sigma_\delta = \sigma^{-1} $
  so Lemma \ref{lemma:non-crossing-inverse} implies that $ \sigma $ has only cycles of length 1 or 2, hence part (1) is proved.
  
  Suppose now that $ d_1 \rightarrow \infty $ and $ d_2 $ is fixed. Then only $ \pi $ such that $ g(\pi ) = 0 $ will contribute to the limit in the summation (\ref{p-t:moments}). Applying again Lemma \ref{lemma:pi_delta_invariance}, this is equivalent to $ \pi = \sigma \delta \sigma^{-1} $, for $ \sigma $ a non-crossing permutation on $ [ n ] $. In this case, we have that
  $
  g(\delta \pi\delta) = g( \delta \sigma \delta \sigma^{-1} \delta).
  $
  
   Also,
   $
   \# ( \gamma \delta \gamma^{-1} \vee \delta \pi \delta ) =
    \frac{1}{2}
    \# ( \gamma \delta \gamma^{-1} \delta \sigma \delta \sigma^{-1}\delta ),
   $
    and,
 if $ k \in [ n ] $, we have that
 \begin{align*}
  \gamma \delta \gamma^{-1} \delta \sigma \delta \sigma^{-1} \delta (k)
   & = 
   \gamma \delta \gamma^{-1} \delta \sigma  (  k )
   =
   \gamma \sigma ( k )
    \\
   \gamma \delta \gamma^{-1} \delta \sigma \delta \sigma^{-1} \delta ( - k)
   & = 
   \gamma \delta \gamma^{-1} \delta \sigma \delta \sigma^{-1}  ( k )
   = 
   \gamma \delta \gamma^{-1} \delta \sigma ( -  \sigma^{-1}  ( k ) )\\
   &  =
   \gamma \delta \gamma^{-1} ( \sigma^{-1} ( k ) )
   =
   \gamma^{-1} ( \sigma^{-1} ( k ) ) = ( \sigma \circ \gamma ) ^{ -1} ( k )
 \end{align*} 
 
  Moreover,
    \[ 
    \# ( ( \delta \pi \delta) \delta ) = \# ( \delta \pi) = \# ( ( \pi \delta)^{ -1} )  = 2 \# (\sigma ),
    \]
    hence, Lemma \ref{lemma:bn} gives that
    \[
    g( \delta \pi \delta ) = \# ( \gamma \sigma ) + \# ( \sigma ) - ( n + 1) = \#(\sigma) + e(\sigma) - n \]
    so equation (\ref{p-t:moments}) becomes
\[
\lim_{d_1 \rightarrow \infty}
    E \circ \tr \otimes \tr \big( (W^\Gamma)^n  \big) 
      = \sum_{\sigma \in S_{NC}(n)}
      c^{\#( \sigma)}
       d_2^{ e (\sigma) + \# ( \sigma ) - n } .
\]
Thus
\[
\lim_{d_1 \rightarrow \infty}
    E \circ \tr \otimes \tr \big( (d_2W^\Gamma)^n  \big) 
      = \sum_{\sigma \in S_{NC}(n)}
      ( c d_2)^{\#( \sigma)}
       d_2^{ e (\sigma)} 
= \sum_{\sigma \in S_{NC}(n)} \kappa_\sigma.
\]
where $\kappa_n = c d_2$ for $n$ odd and $\kappa_n = c d_2^2$ for $n$ even. The conclusion follows because $\kappa_n = (cd_2) \frac{d_2+1}{2} + (-1)^n (cd_2) \frac{d_2 - 1}{2}$ (see the proof of Theorem \ref{thm:distrib:2}). The case $ d_1 $ fixed and $ d_2 \rightarrow \infty $ is similar.  
    
  \end{proof}

 \begin{theorem}
 If both $ d_1, d_2 \rightarrow \infty $, then $\{ W, W^\rtr\}$ is an asymptotically free family.
 \end{theorem} 
  
 \begin{proof}
 The result is a consequence of Theorems \ref{thm:main_real} and \ref{distrib-real}.

 \end{proof}

 \begin{remark}
 For $ W $ a real Wishart random matrix, $W^\rtr $ is not asymptotically free from $ W $ if $ d_1 $ is fixed or  if $ d_2 $ is fixed.
 \end{remark} 
  
 Indeed, for $ n =2 $ and $ \epsilon_1 = 1 $ and $ \epsilon _2 = -1 $, the formula from Theorem \ref{thm:24} gives 
 
\[
 E \circ \tr \otimes \tr \big(  W W^\rtr \big) 
 = 
 \sum_{\pi \in \cP_2(\pm 2)} 
 \Big( \frac{p}{d_1 d_2}\Big)^{\#(\pi\delta)/2}
 d_1^{g(\pi)} d_2^{g(\epsilon \pi \epsilon)}.
\]  
  There are only 3 pairings in $ P_2( \pm 2 ) $: $ \pi_1 = (1, -1 ) , ( 2, -2 ) $,
   $ \pi_2 = (1, 2 ),( -1, -2 ) $
 and
  $ \pi_3 = ( 1, -2 ), ( -1, 2) $.
  Direct calculations give that
   $ \pi_1 \delta = \textrm{id} $, 
   $ \pi_2 \delta = (1, -2), (-1, 2) $ and 
   $ \pi_2 \delta = (1, 2) , (-1, -2) $.

 Moreover, $ \epsilon \pi_1 \epsilon = \pi_1 $, while
 $ \epsilon \pi_2 \epsilon = \pi_3 $ and $ \epsilon \pi_3 \epsilon = \pi_2 $; 
 also, for $ n =2 $, we have that
  $ \gamma \delta \gamma^{-1} 
  = (1, -2), ( -1, 2).$  Therefore $ g(\pi_1) =g(\pi_2) = 0 $ and $ g(\pi_3 ) = 1$, so
  \[
   E \circ \tr \otimes \tr \big(  W W^\rtr \big) 
   = 
   \Big( \frac{p}{d_1 d_2}\Big)^2 + 
   \Big( \frac{p}{d_1 d_2}\Big) \cdot
   \Big( \frac{1}{d_1} + \frac{1}{d_2} \Big) 
  \]
 and the second term in the equation above does not cancel asymptotically unless both $ d_1, d_2 \rightarrow \infty $.


\thebottomline
\end{document}